\documentclass{article}
\usepackage{amsmath} %ÍÆ¼öÊ¹ÓÃµÄÊýÑ§¹«Ê½ºê°ü£¬ÒòÎª¿ÉÒÔÓÃalign»·¾³¸üºÃµØÅÅ°æÐÐ¼ä¹«Ê½
\usepackage{amsfonts}
\usepackage{amssymb,amscd,amsthm}
\usepackage{geometry}
\usepackage{mathrsfs}
\usepackage{graphicx}
\usepackage{ragged2e}
\usepackage{cases}
\usepackage{amsthm} %²åÈë¶¨Àí¡¢Ö¤Ã÷µÄºê°ü
\usepackage{enumitem} %²åÈëÁÐ¾ÙÏîÄ¿µÄºê°ü
\usepackage{mathtools}
\usepackage{extarrows}
\usepackage{setspace}

\newcommand{\RNum}[1]{\expandafter{\romannumeral #1\relax}}

\numberwithin{equation}{section}
\newcommand\asertion[1]{ssertion $({\mathrm{\romannumeral #1\relax}})$}%假设1
\newcommand\infer[2]{$(\mathrm{\romannumeral #1\relax})\implies (\mathrm{\romannumeral #2\relax})$}
\newcommand{\algmod}{algebraic module}
\newcommand{\algbmod}{algebraic bimodule}

\newcommand{\ass}{associator}%定义结合子
\newcommand{\asselm}{associative element}%定义结合元
\newcommand{\decomp}{cyclic decomposition}%定义可乘
\newcommand{\tezh}{characteristic vector}
\newcommand{\tezzh}{characteristic value}\newcommand{\Hom}{\text{Hom}}
\newcommand{\End}{\text{End}}
\newcommand\conjgt[1]{\overline{#1}}%交换结合元集

\newcommand{\bfs}{Without loss of generality we can assume}%定义可裂
\def\O{\mathbb{O}}
\def\S{\mathbb{S}}
\def\R{\mathbb{R}}

\def\C{C\ell_7}

\newcommand\clifd[1]{C\ell_{#1}}%结合元集

\newcommand{\spf}{\mathbb{F}}
\newcommand{\spc}{\mathbb{C}}
\newcommand{\spr}{\mathbb{R}}
\newcommand{\spo}{\mathbb{O}}%八元数
\newcommand\huaa[1]{\mathscr{A}(#1)}%结合元集
\newcommand\hua[3]{\mathscr{#1}^{#2}(#3)}%循环元集

\newcommand\huac[1]{\mathscr{C}(#1)}%循环元集
\newcommand\huacc[2]{\mathscr{C}^{#1}(#2)}%循环元集

\newcommand\pureim[1]{\mathit{Im}(#1)}%纯虚元集
\newcommand{\re}{\mathit{Re}\,}%
\newcommand{\vre}{\mathit{Re}\,}%
\newcommand{\imaginary}{\mathit{Im}\,}

\newcommand\fx[2]{\left<#1,#2\right>}%
\newcommand\dbb[3]{\lfloor#1,#2,#3\rfloor }

\newcommand\generat[1]{\left\langle #1\right\rangle}%

\newtheorem{mydef}{Definition}[section]
\newtheorem{rem}[mydef]{Remark}
\newtheorem{eg}[mydef]{Example}
\newtheorem{cor}[mydef]{Corollary}
\newtheorem{prop}[mydef]{Proposition}
\newtheorem{lemma}[mydef]{Lemma}
\newtheorem{thm}[mydef]{Theorem}

\newtheorem{step }[stp]{Step }

\begin{document}
\title{ Octonionic bimodule} %Ìí¼Ó±êÌâ
\author{Qinghai Huo, Guangbin Ren} %Ìí¼Ó×÷Õß
\date{} %LaTeX»á×Ô¶¯Éú³ÉÈÕÆÚ£¬Èç¹û²»ÐèÒª¾Í¼ÓÕâÒ»²½½«ÈÕÆÚÈ¥µô
\maketitle 
\begin{abstract}
	%The aim of this paper is to provide a foundational investigation of $\spo$-bimodule.
The structure of octonionic bimodules is formulated in this paper. It turns out that every octonionic bimodule is a tensor product,  the category of octonionic bimodules is isomorphic to the  category of real vector spaces.  We show that there is also a real part structure on  octonionic bimodules similar  to the quaternion case. 
%It turns out that the category of octonionic bimodules is isomorphic to the  category of real vector spaces, too.
Different from the quaternion setting , the octonionic bimodule sturcture is uniquely determined by its left module structure and hence the real part can be obtained only by left multiplication. 
The structure of  octonionic submodules generated by one element is more involved, which leads to many obstacles to further development of the octonionic functional analysis. 
We introduce a notion of \decomp\ to deal with this difficulty. Using this concept, we   give a complete description of the submodule generated by one element in octonionic bimodules. This paper clears the  barrier of the structure of $\O$-modules for the later study of  octonionic functional analysis. 
\end{abstract}
%\vspace{1ex}\smallis equivalent to theIn fact, we have proved much more.There are explicitly $8$ kinds of elements in terms of its \decomp.

\tableofcontents
\section{introduction}

Theory of quaternion Hilbert spaces  has been studied a lot (\cite{horwitz1993QHilbertmod,razon1992Uniqueness,razon1991projection,soffer1983quaternion,viswanath1971normal}). The Theory of spherical spectrum of normal operator, continuous slice functional caculus in quaternionic Hilbert spaces had  been also established  (\cite{ghiloni2013slicefct}). The quaternionic vector sapces have also been studied 
thoroughly (\cite{ng2007quaternionic}), which makes it well-grounded to do further study on the quaternionic functional analysis. However, in the octonion case, the structure of the one-sided modules and bimodules over octonion are not completely clear, this leads to many obstacles to further development of the octonionic functional analysis, although there are also some results on the study of octonion Hilbert spaces (\cite{goldstine1964hilbert,goldstine1964hilbert2,ludkovsky2007algebras,ludkovsky2007Spectral}). 
%The space under considered in the previous study of octonion Hilbert spaces theory is always choosen to be a left (or right) $\O$-module. However, when considering the  global property of operator spaces, such as the left (or right) multiplication of an operator which should induce an  $\O$-module structure on the set of linear operators, and the concept of Banach dual operator of octonion version, one naturaly needs to require the space considered to be an $\O$-bimodule. 
%, for example, see \cite{goldstine1964hilbert}
%When considering  the Riesz Representation Theorem of octonion version, one should give the dual space an $\O$-module structure. ; and the category of (both one-sided and two sided) quaternion Hilbert spaces is  equivalent to the category of real Hilbert spaces
 Consequently, it is worth  discussing the structure of  $\O$-modules.
 
  The general case of  the structures of bimodules over Jordan algebra and alternative algebra has been  studied by Jacobson in \cite{jacobson1954structure}.  The left-alternative left-modules for the real algebra of octonions have been considered, and  the irreducible ones are known to be isomorphic to the regular or conjugate regular modules (See, for example, Chapter 11 of the monograph by Zhevlakov, Slinko, Shirshov, and Shestakov \cite{zhevlakov1982Rings}).  And we point out that the $\O$-vector space studied in \cite{ludkovsky2007algebras} is not an $\O$-bimodule under  the definition given in \cite{jacobson1954structure}. It is actually a left $\O$-module with an irrelevant right $\O$-module structure. And it seems no proof for the structure of such $\O$-vector spaces being a tensor product, which has been used  several times. 
 
In quaternion case,  Ng  gives a systematic study in \cite{ng2007quaternionic}, which shows the category of quaternion vector spaces, that is, quaternion bimodules, is equivalent to the category of real vector spaces. More precisely, there is a natural sturcture of real part on each quaternion bimodule, which is the corresponding real vector space. A natural question is whether  similar results hold in octonion case.
In previous work \cite{liyong2019octonionmodule}, we have formulated  the structures  of left $\O$-modules. It shows that there is an isomorphism
between the category $O$-$\mathbf{Mod}$  and the category $\C$-$\mathbf{Mod}$, here the object in $O$-$\mathbf{Mod}$ is left $\O$-module. Each    left $\O$-module $M$ will be of the form   $$M=\spo\huaa{M}\oplus {\spo}\hua{A}{-}{M}.$$ 
Where $\huaa{M}$ and $\hua{A}{-}{M}$ represent  the subset of {associative elements} and {conjugate associative elements} respectively:
%Where $\huaa{M} $ is the set of  all associative elements:
$$\huaa{M}:=\{m\in M\mid [p,q,m]=0,\ \forall p,q \in \O\};$$
and
%$\hua{A}{-}{M}$ is the set of all conjugate associative elements:
$$\hua{A}{-}{M}:=\{m\in M\mid (pq)m=q(pm),\forall p,q\in \spo\}.$$
 It is therefore  natural to ask whether a given left $\O$-module  admits a compatible bimodule structure, and if so, is it unique? if not, what is the condition for a left $\O$-module to admit a compatible bimodule structure. In this paper, we direct ourselves to answering these questions.

In this paper, we show that the necessary and sufficient condition for a left $\O$-module  admiting a compatible bimodule structure is just the vanishing of the subset of conjugate associative elements.
 And if so,  the bimodule structure is then uniquely determined by its left multiplication.
  More precisely, we obtain:
\begin{thm}
A left $\O$-module $M$ admits a compatible $\spo$-bimodule structure if and only if it holds $M=\O \huaa{M}$. 
	
	Moreover, if $M$ admits an $\spo$-bimodule structure, then it is unique.
	%in  this case, the right multiplication is just given by 
	%$$xp:=px,\quad \forall p\in \O,\;\forall x\in \huaa{M}.$$
\end{thm}

There is also a  structure of real part on $\O$-bimodules as in quaternion case:
$$\re x=\frac{5}{12}x-\frac{1}{12}\sum_{i=1}^7 e_ixe_i.$$ 
And we can rewrite this formula in terms of left multiplication:
$$ \re x=x+\dfrac{1}{48}\sum _{i,j,k}\epsilon_{ijk}e_i[e_j,e_k,x],$$
where the symbol $\epsilon_{ijk}$ depends on the multiplication table of the octonions.
Using this one easily obtains that an $\O$-bimodule $M$ is isomorphic to the  tensor product $\re M\otimes \O$, coherent with the quaternion case. Therefore, we get that the category of $\O$-bimodules is isomorphic to the category of $\R$-vector spaces.
%\begin{thm}
%	The category of $\O$-bimodules is isomorphic to the category of $\R$-vector spaces.
%\end{thm}

%In previous work \cite{liyong2019octonionmodule}, we have introduced the notion of cyclic element and proved that $M=\text{Span}_\R\huac{M}$. 

 Our last topic is about the structure of submodules generated by one element. 
 In contrast to the complex or quaternion setting,   some new phenomena occur in the setting of octonions, which has already  been known in \cite{goldstine1964hilbert}: If $m$ is an  element of an octonionic module, then
 
 $\bullet$   $\O m$ is not an submodule in general. 
 
 $\bullet$   The  submodule generated by $m$ maybe the whole   module.
 
 This means that the structure of  octonionic submodules is more involved and such property is crucial for classical functional analysis.  We point out that some gaps appear  in establishing the octonionic version of Hahn-Banach  Theorem by taking $\O m$ as a submodule   (\cite[Lemma 2.4.2]{ludkovsky2007algebras}).
 The submodule  generated by a submodule $Y$ and a point $x$ is not of the form $\{y+px\mid y\in Y,\;  p\in \O\}$,  this is wrong even for the case $Y=\{0\}$. It means the  proof can not repeat the way in canonical case. The involved  structure of  octonion submodules  accounts for the slow developments of octonion Hilbert spaces.
  We shall give a new proof in a later paper.

 This phenomena motivates us to introduce a  new notion of  cyclic elements, which play  a key role in the study of  submodules.
An element $m$ in a given  module $M$ is called 
\textbf{cyclic elements} if the submodule generated by it is  exactly $\O m$.
%The collection of these elements is denoted by $\huac{M}$. It turns out that the cyclic elements are determined by the associative elements $\huaa{M}$ and the conjugate associative elements $\hua{A}{-}{M}$ completely.
We next introduce a notion of \decomp\ to describe the structure of these submodules generated by one element. 
It turns out that each element can be decomposed into a sum of some special cyclic elements.
More precisely, we obtain:
\begin{thm}
	Let $m$ be an arbitrary element of an $\O$-bimodule $M$. Then  $$\generat{m}_\O=\bigoplus_{i=1}^n\O m_i,$$
where  $\{m_i\}_{i=1}^n\subseteq \huac{M}$  is an arbitrary \decomp\ of $m$.

\end{thm}
The length of a \decomp\ is therefore an invariant  of $m$ and by definition  at most $8$,  hence there are only 
$8$ kinds of elements in $\O$-bimodules and each element $m$ will generate a submodule with dimension $\dim_\R \generat{m}_\O \leqslant 64$. 

By the way, we point out a  mistake in \cite{ludkovsky2007algebras}. 
%The first is displayed in the proof of Hanh-Banach Theorem (\cite[Theorem 2.4.1]{ludkovsky2007algebras}). Note that therein the definition of $\O$-vector space is not a true bimodule, then the decomposition of the $\O$-vector space may be not of the form:
%$$X=X_0\oplus X_1i_1\oplus\cdots\oplus X_{7}i_7,$$
%where $ i_1,\ldots,i_7 \text{  is  the standard generators of the }\O.$
%Even so, the proof won't work. It is trivial to get the extension which preserves the real linearity, whereas not preserves the norm.  In fact,  the submodule postulate has not been used throughout the proof in \cite{ludkovsky2007algebras}, while this is crucial for proof of  Hanh-Banach Theorem of  the octonion version, we will formulate it in a later paper. 
%Another mistake appears in the  proof of the corollary of Hahn Banach Theorem (\cite[Lemma 2.4.2]{ludkovsky2007algebras}).
%  The submodule  generated by a submodule $Y$ and a point $x$ is not of the form $\{y+px\mid y\in Y, p\in \O\}$, since we have seen in \cite{goldstine1964hilbert,liyong2019octonionmodule} that this is wrong even for the case $Y=\{0\}$. Hence the  proof can not repeat the way in canoical case. We shall give a new proof in a later paper.
It appears in the  proof of the corollary of Hahn Banach Theorem (\cite[Lemma 2.4.2]{ludkovsky2007algebras}),   
which declared every element in an $\O$-module will satisfy $\O x=x\O$. In fact, with the help of the notion of \decomp,  we shall show that only cyclic elements posses such property.  
%\begin{enumerate}[label=\textbf{Step \arabic*}:] 
%	\item hj
%	\item j
%which is introduced to deal with the difficulty of submodule generated by one element,
%\end{enumerate}To deal with the difficulty of submodule generated by one element, w
\section{Pcreliminaries}
In this section, we review some basic  properties of the algebra of the octonions   $\spo$ and
one-sided $\O$-modules, and introduce some fundamental notations.

\subsection{The octonions $\spo$}

%In this  section,  we recall some known classical results in the octonion
The  algebra of the octonions   $\spo$  is  a non-associative, non-commutative, normed division algebra over the $\spr$. Let     $e_1,\ldots,e_7$ be its natural  basis throughout this paper, i.e., $$e_ie_j+e_je_i=-2\delta_{ij},\quad i,j=1,\ldots,7.$$ For convenience, we  denote  $ e_0=1$.

In terms of the natural basis, an element in octonions can be written as $$x=x_0+\sum_{i=1}^7x_ie_i,\quad x_i\in\spr,$$
%where the basis satisfies the rules
%$$e_ie_j+e_je_i=-2\delta_{ij},\quad i,j=1,\ldots,7.$$
The conjugate octonion of $x$ is defined by  $\overline{x}=x_0-\sum_{i=1}^7x_ie_i$, and the norm of $x$ equals $|x|=\sqrt{x\overline{x}}\in \spr$, the real part of $x$ is $\re{x}=x_0=\frac{1}{2}(x+\overline{x})$. \textbf{The term  $\sum_{i=1}^7x_ie_i$ will be  abbreviated as $\sum x_ie_i$ in this paper.} We denote by $\S$ the set of imaginary units in $\O$:
$$\S:=\{J\in \O\mid J^2=-1\}.$$
Then there is a book structure on octonions:$$\O=\bigcup_{J\in \S}\spc_J,$$
here $\spc_J$ represents the complex plane spaned by $\{1,J\}$.

%The full multiplication table is conveniently encoded in the 7-point projective plane, which is often called the Fano mnemonic graph.
%In the Fano mnemonic graph, the vertices are labeled by  $1, \ldots, 7$
%instead of $e_1,\ldots, e_7$. Each of the 7 oriented lines gives a quaternionic triple. The
%product of any two imaginary units is given by the third unit on the unique line
%connecting them, with the sign determined by the relative orientation  (see \cite{wang2014octonion}).

%\
%
%\noindent \small{\textbf{Fig.1} Fano mnemonic graph}
%\begin{flushright}
%	\includegraphics[width=4cm]{Fano}
%\end{flushright}
%\normalsize

The associator of three octonions is defined as $$[x,y,z]=(xy)z-x(yz)$$
for any   $x,y,z\in\spo$, which is alternative in its arguments and has no real part. That is, $\spo$ is an alternative algebra and hence it satisfies the so-called R. Monfang identities \cite{schafer2017introduction}:
$$(xyx)z=x(y(xz)),\ z(xyx)=((zx)y)x,\ x(yz)x=(xy)(zx).$$ The commutator is defined as $$[x,y]=xy-yx.$$
One can prove that, for any $J\in \S$ and   $x\in\spc_J\setminus \spr $, we have 
$$\{p\in \O\mid [p,x]=0   \}=\spc_J.$$
% which is alternative totally antisymmetric such as
%[x,y,z]=−[y,x,z]=[y,z,x]
%in its arguments and has no real part. Although the associator does not vanish in general, the octonions do satisfy a weak form of associativity known as alternativity, namely
%[x,x,y]=0
%and the so-called R. Monfang identities:
%(xyx)z=x(y(xz)),z(xyx)=((zx)y)x,x(yz)x=(xy)(zx).
%The underlying reason for this is that two octonions determine a quaternionic subalgebra of   O , so that any product containing only two octonionic directions is associative.
%\subsection{Convientions}
%In this paper, we shall  use $e_1,\ldots,e_7$ to denote the basis of $\spo$ systematically, the term  $\sum_{i=1}^7x_ie_i$ will be always abbreviated as $\sum x_ie_i$ (that is, when the summation is from $1$ to $7$). 

%\subsubsection{The $\epsilon$-notation}
The full multiplication table is conveniently encoded in the 7-point projective plane, which is often called the Fano mnemonic graph.
In the Fano mnemonic graph, the vertices are labeled by  $1, \ldots, 7$
instead of $e_1,\ldots, e_7$. Each of the 7 oriented lines gives a quaternionic triple. The
product of any two imaginary units is given by the third unit on the unique line
connecting them, with the sign determined by the relative orientation.
%  (see \cite{wang2014octonion}).

\

\noindent \small{\textbf{Fig.1} Fano mnemonic graph}
\begin{flushright}
\centerline{\includegraphics[width=4cm]{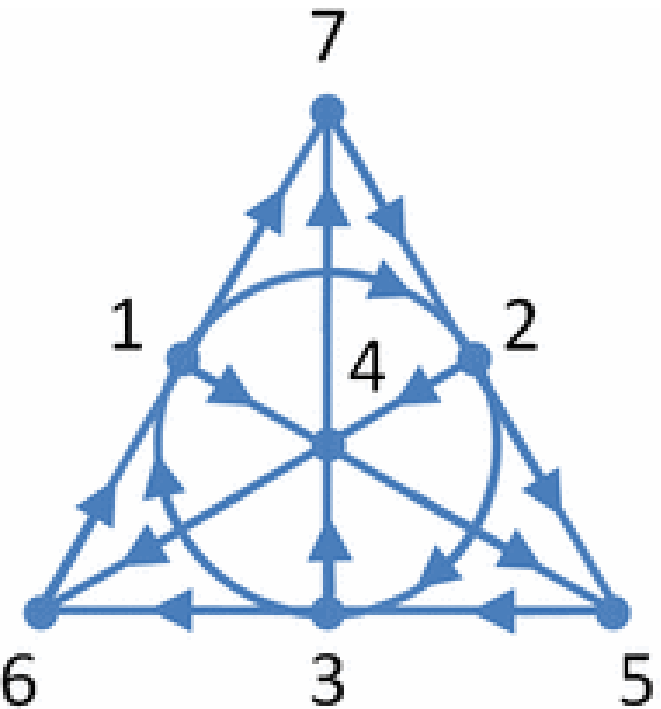}}
\end{flushright}
\normalsize

\

It will be convenient to use an $\epsilon$-notation that will now be introduced 
(see \cite{bryant2003some}). This is the unique symbol that is skew-symmetric in either three or four indices.
% This is the unique symbol that is skew-symmetric in either three or four indices and satisfies
% One way to think of this symbol is via the cross product: $e_i \times e_j = \epsilon_{ijk}e_k$. That is,  $e_ie_j=\epsilon_{ijk}e_k-\delta_{ij}$.
One way to think of this symbol is:
\begin{align}
&e_ie_j=\epsilon_{ijk}e_k-\delta_{ij}\\
&[e_i,e_j,e_k]=2\epsilon_{ijkl}e_l\label{eq:e(ijkl)el=[ei,ej,ek]}
\end{align}
% $$e_ie_j=\epsilon_{ijk}e_k-\delta_{ij},\quad [e_i,e_j,e_k]=2\epsilon_{ijkl}e_l.$$
The symbol $\epsilon$ satisfies various useful identities. For example (using the summation
convention),
\begin{align}
\epsilon_{ijk}\epsilon_{ijl}&=6\delta_{kl}\label{eq:ep(ijk)ep(ijl)=6delta(kl)}\\
\epsilon_{ijq}\epsilon_{ijkl}&=4\epsilon_{qkl}\label{eq:contract id eijq eijkl=4eqkl}\\
\epsilon_{ipq}\epsilon_{ijk}&=\epsilon_{pqjk}+\delta_{pj}\delta_{qk}-\delta_{pk}\delta_{qj}\\
\epsilon_{ipq}\epsilon_{ijkl}&=\delta_{pj}\epsilon_{qkl}-\delta_{jq}\epsilon_{pkl}+\delta_{pk}\epsilon_{jql}-\delta_{kq}\epsilon_{jpl}+\delta_{pl}\epsilon_{jkq}-\delta_{lq}\epsilon_{jkp}
\end{align}
We shall always  use the  Einstein summation convention when we compute in terms of $\epsilon$-notation.

\

The $14$-dimensional group $G_2$ is the smallest of the five exceptional Lie groups and is closely related to the octonions. In particular, $G_2$ can be defined as the automorphism group of the octonion algebra:
$$G_2:=\{g\in GL(\O)\mid g(xy)=g(x)g(y) \text{ for all } x,y\in \O   \}.$$  As a fact, $G_2\subseteq {SO}(7)$. We refer to \cite{harvey1990spinors,salamon2017notesonoctonion} for more details. 

%Now we assemble a few simple properties in $\O$.
%\begin{lemma}\label{lem:ass is pure im}
%We have the following facts:
%	\begin{itemize}
%		\item 	$[p,q,r]\in \pureim{\O}$ for all $p,q,r\in \O$.
%		\item Each non-trivial algebra autmorphism of $\O$ is an isomorphism.
%	\end{itemize}
%
%\end{lemma}
%\begin{proof}
%	 Using $\epsilon$-notation, it's easy to obtain the \ass\ is pure imanary in $\O$. Note that the kernel of a non-trivial algebra autmorphism is a subalgebra of $\O$, this implies it must be injective and hence an isomorphism.
%\end{proof}

%\subsubsection{Associative subspace in $\pureim{\O}$}
We can now use octonion multiplication to define a vector cross product $\times$ on $\R^7$ (\cite{schafer2017introduction}). Given $u,v\in \R^7$, we regard them as elements in $\pureim{\O}$, then $$u\times v:=\imaginary (uv).$$  A $3$–dimensional subspace $\Lambda\subset\pureim{\O}$ is called \textbf{associative} if the associator bracket vanishes on $\Lambda$, i.e.,
$$[u, v, w] = 0 \text{ for all } u, v, w\in \Lambda. $$ 
	If $u, v\in \Lambda$ are linearly independent, then the subspace spanned by the
vectors $u, v, u \times v$ is associative (see \cite{schafer2017introduction}). This subspace will be denoted by $\Lambda(u,v)$. By definition, it is easy to verify: 
$$\Lambda(u,v)=\{x\in \pureim{\O}\mid [u,v,x]=0  \}.$$
\subsection{The definition of $\O$-modules}
%The first notation of bimodules over non-associative rings was introduced by Eilenberg\cite{eilenberg1948extensions} to study non-associative rings.  Jacobson studied the structures of bimodules over Jordan algebra and alternative algebra in his article\cite{jacobson1954structure}. 
%The notion of  one-sided modules over octonion  appears in many literatures as well, for example, \cite{goldstine1964hilbert,liyong2019octonionmodule}. 
It's well known that  the algebra of octonions  $\spo$ is an   alternative algebra, thus that  an $\spo$-module $M$ 
is actually  an alternative module. From now on let $A$ be a unital alternative algebra over a field $\spf$.
 For the sake of completeness, we give the  definition of modules over alternative algebra as follows.

\begin{mydef}\label{def: 1 alternative algebra}

	An $\spf$-vector space $M$ is called a \textbf{left alternative module} over $A$, if there is an  $\spf$-linear map $$L:  A\rightarrow \End_{\spf}M,\quad a\mapsto L_a $$ satisfying  $L_1=id_M$ and 
	$$[a,b,x]=-[b,a,x], \quad \forall a,b\in A, \; x\in M.$$
	Here the \ass\ is defined by $[a,b,x]:=(ab)x-a(bx)$.
	The definition of  right alternative \algmod\ is similar.		
	A left alternative   $A$-module $M$ is called an  \textbf{alternative bimodule} if  the  \ass\ is alternative:
	$$[p,q,m]=[m,p,q]=[q,m,p], $$ 
	$\text{ for all } p,q\in A,  \text{ and for all }m\in M$.
	Where the \textbf{middle \ass} $[q,m,p]$ is defined by $$[q,m,p]:=(qm)p-q(mp),$$
	and the \textbf{right \ass} $[p,q,m]$ is defined by $$[p,q,m]:=(pq)m-p(qm).$$
	% if  the right and middle \ass\ coincide with the left \ass.

\end{mydef}
	\begin{rem}\label{rem:def of alg-mod}
	Let $M$ be an $A$-module.	For all $m,m'\in M$,  $\alpha,\beta\in \spf$ and all $a,a'\in A$, we have: 
	\begin{enumerate}
		\item $L_a\in \End_{\spf}M \Rightarrow a(\alpha m+\beta m')=\alpha (am)+\beta (am')$.
		In particular, $ a(\alpha m)=\alpha( am)$.
		\item $L\in \Hom_{\spf}(A,\End_{\spf}M )\Rightarrow (\alpha a+\beta a')m=\alpha am+\beta a'm $.
		
		In particular, $ (a\alpha) m=(\alpha a)m=\alpha (am)= a(\alpha m)$ and thus we can  write $ a\alpha m$ unambiguously. %i.e. $\rho (a,\alpha)=\rho (\alpha,a)=0$.
		\item Thinking of $M$ as an $\spf$-vector space, the scalar multiplication over $\spf$  coincides with $L\vert_\spf $  since $L(1)=id_M$.
		%\item In fact, the linearity  of $ L$ implies the linearity of $\rho$. and X is thought of as ,Viewed as
		
		%\item Regard $\algma$ as a ring, then $M$ is a $\algma$-module with \ass\ $\rho$ automaticl.
	\end{enumerate}
	
\end{rem}
%Note that this definition  is  equivalent to the definition given in \cite{goldstine1964hilbert,ludkovsky2007algebras}, where the axiom $({\romannumeral3)}$ is replaced by  $$q^2m=q(qm), \text{ for all } q\in \O, m\in M.$$
%The proof is trivial by polarizing the above relation (see \cite{goldstine1964hilbert}). 
%It also agrees with the one given in \cite{goldstine1964hilbert} where $M$  needs to satisfy an additional axiom: $p(p^{-1}x)=x$, which can be deduced from the equation $p(px)=p^2x$ directly . 
Note that 
the left alternativity requirement of the \ass\ in $M$ is  equivalent to the following condition given in \cite{goldstine1964hilbert,ludkovsky2007algebras}
$$a^2m=a(am), \text{ for all } a\in A, \,m\in M.$$The proof is trivial by polarizing the above relation. And the notion of $A$-bimodules here  agrees with the defintion given in \cite{jacobson1954structure}. However, the following postulates:
$$1m=m1=m,\qquad a^2m=a(am),\quad ma^2=(ma)a,\quad (am)a=a(ma)$$
   can not   deduce that $M$ is an alternative bimodule in general.

%There is  an equivalent definition of left alternative \algmod\ (see \cite{jacobson1954structure,goldstine1964hilbert}).
%
%\begin{mydef}\label{def:2 alternative algebra}
%	Let  $M$ be an $\spf$-vector space. We say that $M$ is an (left) \textbf{  alternative \algmod} over $A$ if  there exists an $\spf$-bilinear map: $A\times M\rightarrow M$ denoted by $(r,m)\mapsto rm$, such that  the following axioms hold for all $a\in A,\ m\in M$:
%	$$1m=m,\qquad a^2m=a(am).$$
%\end{mydef}

%\begin{rem}
%%	The above two  definitions are equivalent (the proof is obvious with $a\in A$ in definition \ref{def:2 alternative algebra} replaced by $a+b$).
%	 In general, that  $M$ is  alternative \algbmod\  can not be  deduced from the following condition:
%	$$1m=m1=m,\qquad a^2m=a(am),\quad ma^2=(ma)a,\quad (am)a=a(ma).$$
%	%similar definition of bimodule as   definition \ref{def:2 alternative algebra}.% More precisely, we can define week alternative \algbmod\ as follows.
%	
%\end{rem}

\

One useful identity  which holds in any left module $M$ is 
\begin{equation}\label{eq:[p,q,r]m+p[q,r,m]=[pq,r,m]-[p,qr,m]+[p,q,rm]}
[p,q,r]m+p[q,r,m]=[pq,r,m]-[p,qr,m]+[p,q,rm].
\end{equation}
Here $m\in M$ is an arbitrary element and it holds for all $p,q,r\in A$.

The Moufang identities and Artin Theorem hold as before.
\begin{thm}[\textbf{Moufang identities}]
	Let $M$ be an alternative \algbmod\ over $A$. Then for all $p,q\in A,m\in M$, the Moufang identities hold:
	\begin{gather}
	(pmp)q=p(m(pq))\\
	q(pmp)=((qp)m)p\label{eq:moufang2 q(pmp)=((qp)m)p}\\
	(pq)(mp)=p(qm)p.
	\end{gather}
\end{thm}
\begin{proof}
	The proof  is similar as in  classical case. 
	We only prove the first identity.
	\begin{align*}
	(pmp)q-p(m(pq))&=[pm,p,q]+(pm)(pq)-p(m(pq))\\
	&=[pm,p,q]+[p,m,pq]\\
	&=-[p,pm,q]-[p,pq,m]\\
	&=-(p^2m)q+p((pm)q)-(p^2q)m+p((pq)m)\\
	&=-p^2(mq)-[p^2,m,q]-p^2(qm)-[p^2,q,m]+p((pm)q+(pq)m)\\
	&=p([p,m,q]+[p,q,m])\\
	&=0
	\end{align*}
The rest proof runs as classical case.	
\end{proof}

%Artin 定理
\begin{thm}[\textbf{Artin Theorem}]
	Let $M$ be a left alternative \algmod\ over $A$. Then $[p^m,p^n,x]=0$, for all $p\in A$,  any $m,n\in \mathbb{N}$ and  $x\in M$.
	
\end{thm}
\begin{proof}
	The proof will be divided into two steps.
	\begin{step }
		$[p^m,p,x]=0,\ \forall p\in A,\forall m\in \mathbb{N},\forall x\in M$.
	\end{step }
			Clearly it holds for $m=1$. Assume the formula holds for degree $k$, we will prove it for $k+1$.
		By induction hypothesis, $p(p^{k+1}x)=p(p^k(px))=p^{k+1}(px)$,
		hence
		\begin{align*}
		[p^{k+1},p,x]=p^{k+2}x-p^{k+1}(px)=p^{k+2}x-p(p^{k+1}x)=[p,p^{k+1},x]	
		\end{align*}
		By definition of alternative \algbmod, we thus conclude that $[p^{k+1},p,x]=0$.
		%		\begin{align*}
		%		[p^{k+1},p,x]=-[p^{k+1},x,p]&=p^{k+1}(xp)-(p^{k+1}x)p\\
		%		&=(p^kp)(xp)-(p^k(px))p\\
		%		&\xlongequal{\eqref{eq:moufang2 q(pmp)=((qp)m)p}} (p^kp)(xp)-p^k(pxp)\\
		%		&=[p^k,p,xp]=0\\		
		%		\end{align*}
		%The last equlity has used the induction hypothesis.
	\begin{step }
		$[p^m,p^n,x]=0,\forall p\in A,\forall m,n\in \mathbb{N},\forall x\in M$.
		
	\end{step }
	Fix $m$, we prove this by induction on $n$.  We have just proved for case $n=1$. Assume the formula holds for degree $n=k$, we will prove it for $n=k+1$.
\begin{align*}
p^m(p^{k+1}x)&=	p^m((p^{k}p)x)&\text{by Step 1}\\
&=p^m(p^k(px))\\
&=(p^mp^k)(px)\\
&=p^{m+k}(px)&\text{by Step 1}\\
&=p^{m+k+1}x
\end{align*}
That is, $[p^m,p^{k+1},x]=0$. 	This proves the theorem.

%	\textbf{Step 1.} $[p^m,p,x]=0,\ \forall p\in A,\forall m\in \mathbb{N},\forall x\in M$.
%	
%	Clearly it holds for $m=1$. Assume the formula holds for degree $k$, we will prove it for $k+1$.
%	By induction hypothesis, $p(p^{k+1}x)=p(p^k(px))=p^{k+1}(px)$,
%	hence
%	\begin{align*}
%	[p^{k+1},p,x]=p^{k+2}x-p^{k+1}(px)=p^{k+2}x-p(p^{k+1}x)=[p,p^{k+1},x]	
%	\end{align*}
%	By definition of alternative \algbmod, we thus conclude that $[p^{k+1},p,x]=0$.
	%		\begin{align*}
	%		[p^{k+1},p,x]=-[p^{k+1},x,p]&=p^{k+1}(xp)-(p^{k+1}x)p\\
	%		&=(p^kp)(xp)-(p^k(px))p\\
	%		&\xlongequal{\eqref{eq:moufang2 q(pmp)=((qp)m)p}} (p^kp)(xp)-p^k(pxp)\\
	%		&=[p^k,p,xp]=0\\		
	%		\end{align*}
	%The last equlity has used the induction hypothesis.
%	
%	
%	\textbf{Step 2.} $[p^m,p^n,x]=0,\forall p\in A,\forall m,n\in \mathbb{N},\forall x\in M$.
%	
%	
%	Fix $m$, we prove this by induction on $n$.  We have just proved for case $n=1$. Assume the formula holds for degree $n=k$, we will prove it for $n=k+1$.
%	\begin{align*}
%	p^m(p^{k+1}x)=	p^m((p^{k}p)x)&\xlongequal{Step 1}p^m(p^k(px))\\
%	&=(p^mp^k)(px)\\
%	&=p^{m+k}(px)\\
%	&\xlongequal{Step 1}p^{m+k+1}x
%	\end{align*}
%	That is, $[p^m,p^{k+1},x]=0$. 

\end{proof}

\

Our previous work gives a complete classification of left $\O$-modules  \cite{liyong2019octonionmodule}. 
The irreducible ones are already known to be isomorphic to the regular or conjugate regular modules \cite{zhevlakov1982Rings}. However, using the relation between octonions to Clifford algebra, we can give a more simple proof and classify the structure of left $\O$-modules completely.

It is well-known (for example, \cite{baez2002octonions,harvey1990spinors}) that the octonions have a very close relationship with spinors in $7,8$ dimensions. In particular, multiplication
by imaginary octonions is equivalent to Clifford multiplication on spinors in $7$ dimensions.
%We shall consider the relationship between $\spo$-module and Clifford module.
%Note that $[e_i,e_j,x]=-[e_j,e_i,x]$ yields $L_{e_i}L_{e_j}+L_{e_j}L_{e_i}=-2\delta_{ij}$, which leads to a $\clifd{7}$-module structure on $M$, where $\clifd{7}$ is the universal Clifford algebra over $\spr^7$ (see \cite{gilbert1991clifford}).
It follows that 	the category of left $\spo$-modules is isomorphic to the category of left $\clifd{7}$-modules.
Note that $\clifd{7}$ is a semi-simple algebra, we thus obtain that there are only two kinds of irreducible left $\O$-module.
They are the regular  module $\O$ and the conjugate regular module $\overline{\spo}$. Where the left module structure of $\overline{\spo}$ is defined by 
 $$p\hat{\cdot}x:=\overline{p}x,$$
  for all $p\in \spo$, and all $x\in \spo.$
  The \ass\ on $\conjgt{\O}$ is as follows:
 \begin{eqnarray}\label{eq:[]O-}
 [p,q,x]_{\overline{\O}}=[p,q,x]+\overline{[p,q]}x.
 \end{eqnarray}
 The direct sum of their   several copies exhaust all
 octonion modules with finite dimensions. 
 % $$[p,q,x]_{\overline{\O}}=[p,q,x]+\overline{[p,q]}x.$$
 The structure of  general left $\O$-modules is then clear:
  \begin{thm}
 	Let $M$ be a left $\O$-module. Then  $$M\cong\spo\huaa{M}\oplus {\spo}\hua{A}{-}{M}.$$ 
 \end{thm} 
Where $\huaa{M} $ is the set of  all associative elements:
$$\huaa{M}:=\{m\in M\mid [p,q,m]=0,\ \forall p,q \in \O\};$$
$\hua{A}{-}{M}$ is the set of all conjugate associative elements:
$$\hua{A}{-}{M}:=\{m\in M\mid (pq)m=q(pm),\forall p,q\in \spo\}.$$
Its proof will depend on the following lemma. 
\begin{lemma}\label{lem:<m> is finite dim}
	Let $M$ be a left $\spo$-module, then $\generat{m}_\spo$ is  finite dimensional for any $m\in M$. More precisey, the dimension is at most $128$. 
\end{lemma}
\begin{proof}
	$\generat{m}_\spo$ is such module generated by $e_{i_1}(e_{i_2}(\cdots (e_{i_n}m)))$, where $i_k\in\{1,2,\ldots,7\}, n\in \mathbb{N}$. Note that
	$$e_i(e_jm)+e_j(e_im)=(e_ie_j+e_je_i)m=-2\delta_{ij}m,$$
	hence the element defined by $e_{i_1}(e_{i_2}(\cdots (e_{i_n}m)))$ for $n>7$ can be reduced.
	Thus the vectors $\{m,\; e_1m,\;\ldots\;,\;e_7m,e_1(e_2)m,\;\ldots\;,e_1(e_2(\cdots (e_7m)))\}$  will generate $\generat{m}_\spo$,  we conclude that $\dim_\spr \generat{m}_\spo\leqslant C_7^0+C_7^1+\cdots+C_7^7=128$.
\end{proof}
\begin{rem}
	In fact, this property has already appeared in \cite{goldstine1964hilbert}. However, it is worth stressing the essentiality of this property. It  enables us to characterize  the structure of general left $\O$-modules in terms of finite dimensional case, which is already clear in view of the structure of the  $\clifd{7}$-modules.
\end{rem}

\section{Bimodule structure on  $\O$-modules }
As  shown in previous work \cite{liyong2019octonionmodule}, each left $\O$-module $M$ is some copies of $\O$
 and $\overline{\O}$.  It is   natural to ask whether a given left $\O$-module  admits a compatible $\O$-bimodule structure, and if so,  is it unique?  if not, what is the condition for a left $\O$-module to admit a compatible bimodule structure. In this section, we direct ourselves to answering these questions.
 
\subsection{Bimodule structure in low dimensions}% $2$-dimension case:
In this subsection, we are concerned with the $\O$-bimodule structures in low dimensional cases. The general case will be proved in a similar way in the sequel.   
We begin this subsection by proving a technical lemma which is  useful later.
%  Several lemmas will be established which is useful in the sequel.

\begin{lemma}\label{lem:almost linear}
	Let $f\in \End_\R(\O)$. Then the following are equivalent:
	\begin{enumerate}
		\item $\re \big(f(px)-pf(x)\big)=0$ for all $p,x\in \O$.
		\item $f(x)=f_0(x)-\sum e_if_0(e_ix)$, where $f_0(x)=\re f(x)$.
		\item There exists an octonion $ q\in \O $, such that $f(x)=xq$.
		
	\end{enumerate}

\end{lemma}
\begin{proof}
	We prove \infer{1}{2}. Suppose $f(x)=f_0(x)+\sum e_if_i(x)$, where $f_j(x)\in \R,\ j=0,1,\ldots,7$. Using $\epsilon$-notation, we have
	\begin{align*}
	e_if(x)&=e_if_0(x)+e_i\sum e_jf_j(x)\\
	&=-f_i(x)+e_if_0(x)+\sum \epsilon_{ijk}e_kf_j(x)
	\intertext{and}
	f(e_ix)&=f_0(e_ix)+\sum e_jf_j(e_ix).
	\end{align*}  
%	$$f(e_ix)=f_0(e_ix)+\sum e_jf_j(e_ix)$$ and $$e_if(x)=e_if_0(x)+e_i\sum e_jf_j(x)=-f_i(x)+e_if_0(x)+\sum \epsilon_{ijk}e_kf_j(x).$$. 
It follows from   a\asertion{1} that  $\mathit {Re}\big(e_if(x)-f(e_ix)\big)=0$,  we infer that $f_i(x)=-f_0(e_ix)$. Thus a\asertion{2}  holds.

We prove \infer{2}{3}. Denote by  $\fx{\cdot}{\cdot}_\R$ the real inner product on $\O\cong \R^8$, and define $$\fx{x}{y}_\spo:=x\conjgt{y}.$$
By straight-forward calculation, we obtain $$\re \fx{x}{y}_\spo=\fx{x}{y}_\spr.$$ It follows that, $$\fx{e_ix}{y}_\spr=\re\big((e_ix)\conjgt{y}\big)=\re\big(e_i(x\conjgt{y})\big)=\re\big(e_i\fx{x}{y}_\spo\big).$$
This immediately implies $$ \fx{x}{y}_\spo= \fx{x}{y}_\spr-\sum e_i \fx{e_ix}{y}_\spr.$$
Thinking of $(\O,\fx{\cdot}{\cdot}_\R)$ as a real Hilbert space and  $f_0$  a real linear functional,  it follows from the Riesz Representation Theorem that, there exists an element  $y\in \O$ such that $f_0(x)=\fx{x}{y}_\spr$. Therefore by a\asertion{2},
$$f(x)= \fx{x}{y}_\spr-\sum e_i \fx{e_ix}{y}_\spr=x\overline{y}.$$
Setting $q=\overline{y}$, then $f(x)=xq$ as desired.

 We prove \infer{3}{1}. Note that the \ass\ is pure imaginary in $\O$, therefore
 \begin{align*}
 f(px)-pf(x)=(px)q-p(xq)=[p,x,q]\in \pureim{\O}.
 \end{align*}
This completes the proof.
\end{proof}
\begin{rem}
	We introduce a new notion of linearity in the theory of  octonion functional analysis utilizing  a\asertion{1} in a later paper. It turns out that this concept	plays a  role of ``linearity'' as   in the classical
	theory. 
%	in the theory of  octonion functional analysis  as ``linearity'' in the classical
%	 theory.
\end{rem}
 Utilizing this lemma, we can determine the   bimodule structure on  $\spo^2$.
 % endowed with usual left $\O$-module structure
\begin{thm}\label{thm:O2 bimod}
	Let the left $\spo$-module structure of $\spo^2$ is 
	$$p(x,y)=(px,py)\quad \text{ for all }p\in \O, (x,y)\in \O^2.$$
	Then there is a unique compatible bimodule structure on $\spo^2$.% Moreover, the bimodule is \kelie.
\end{thm}

\begin{proof}
	Suppose the right multiplication is given by $$(x,0)\cdot p=(f_p(x),g_p(x));\quad (0,x)\cdot p=(h_p(x),l_p(x)),$$ and hence $$(x,y)\cdot p=\big(f_p(x)+h_p(y),g_p(x)+l_p(y)\big).$$
	
%	\begin{enumerate}
%		\item Claim: $f_p,g_p,h_p,l_p\in End_\spr(\spo), \ \forall p\in \spo$, and are all real linear on $p$.
\begin{step }
For all $p\in \O$,	$f_p,g_p,h_p,l_p\in End_\spr(\spo)$, and they  are  also real linear on $p$.
\end{step }	
		In view of Remark \ref{rem:def of alg-mod}, we have that for all $p\in \spo$,  and all $r\in \spr$,
		$$\big((x,y)\cdot p\big)r=(x,y)\cdot(rp)=\big((x,y)r\big)\cdot p$$
	that is,
	\begin{align*}
	\big(r(f_p(x)+h_p(y)),r(g_p(x)+l_p(y))\big)=\big(f_{rp}(x)+h_{rp}(y),g_{rp}(x)+l_{rp}(y)\big)
	= \big(f_p(rx)+h_p(ry),g_p(rx)+l_p(ry)\big)	.
	\end{align*}
	Let $x$ and $y$ equal zero respectively, we obtain the conclusion.
	%\item $[p,q,(x,y)]=[q,(x,y),p]$.
\begin{step }
Fulfilling the condition $[p,q,(x,y)]=[q,(x,y),p]$.
\end{step }	 
	 
	 In order to obtain a compatible bimodule structure, firstly we must have 
	 \begin{eqnarray}\label{eq:bimod 1id [p,q,(x,y)]=[q,(x,y),p]}
	 [p,q,(x,y)]=[q,(x,y),p]
	 \end{eqnarray}
	We compute:
	\begin{align*}
	[q,(x,y),p]&=(qx,qy)\cdot p-q\big(f_p(x)+h_p(y),g_p(x)+l_p(y)\big)\\
	&=\big(f_p(qx)+h_p(qy),g_p(qx)+l_p(qy)\big)-\big(qf_p(x)+qh_p(y),qg_p(x)+ql_p(y)\big)\\
	&=\big(f_p(qx)-qf_p(x)+h_p(qy)-qh_p(y),g_p(qx)-qg_p(x)+l_p(qy)-ql_p(y)\big).
	\end{align*}
	Let $y=0$, the equation \eqref{eq:bimod 1id [p,q,(x,y)]=[q,(x,y),p]} becomes

	\begin{numcases}{}
	f_p(qx)-qf_p(x)=[p,q,x]\label{eq:proof fp}\\
	g_p(qx)-qg_p(x)=0\label{eq:proof gp}
	\end{numcases}
		Let $x=0$, the equation \eqref{eq:bimod 1id [p,q,(x,y)]=[q,(x,y),p]} becomes
	
	\begin{numcases}{}
	h_p(qy)-qh_p(y)=0\label{eq:proof hp}\\
	l_p(qy)-qg_p(y)=[p,q,y]\label{eq:proof lp}
	\end{numcases}
	Equations \eqref{eq:proof gp} and \eqref{eq:proof hp} imply that $g_p,h_p\in End_\spo(\spo)$, it is easily seen that $Hom_\spo(\spo,\spo)\cong \spr$,  we thus can assume
	$$g_p(x)=r_px,h_p(x)=s_px,\quad r_p,s_p\in \spr.$$
	In view of Lemma \ref{lem:almost linear}, equations \eqref{eq:proof fp} and \eqref{eq:proof lp} ensure us to assume 
	$$f_p(x)=x\tilde{p},l_p(x)=x\hat{p}, \quad \tilde{p},\hat{p}\in \spo.$$
	Since $(x,y)\cdot 1=(x,y)$, we conclude $r_1=s_1=0$, $\tilde{1}=\hat{1}=1$.
	
	%\item $[p,q,(x,y)]=[(x,y),p,q]$above equation \eqref{eq:bimod 2id [p,q,(x,y)]=[(x,y),p,q]}  becomes:
\begin{step }
	Fulfilling the condition $[p,q,(x,y)]=[(x,y),p,q]$.
\end{step }	
	In order to obtain a compatible bimodule structure,  we need the following equation as well:
	\begin{eqnarray}\label{eq:bimod 2id [p,q,(x,y)]=[(x,y),p,q]}
	[p,q,(x,y)]=[(x,y),p,q]
	\end{eqnarray}
	for all $p,q\in \O$.
		We compute:
		\begin{align*}
		[(x,y),p,q]&=\big(f_p(x)+h_p(y),g_p(x)+l_p(y)\big)\cdot q-(x,y)\cdot (pq)\\
		&=\big(x\tilde{p}+s_py,r_px+y\hat{p}\big)\cdot q-\big(x\widetilde{(pq)}+s_{pq}y,r_{pq}x+y\widehat{(pq)}\big)\\
		%&=\Big(f_q\big(f_p(x)+h_p(y)\big)+h_q\big(g_p(x)+l_p(y)\big),g_q\big(f_p(x)+h_p(y)\big)
		 %+l_q\big(g_p(x)+l_p(y)\big)\Big)-\big(f_{pq}(x)+h_{pq}(y),g_{pq}(x)+l_{pq}(y)\big)
		 &=\big((x\tilde{p}+s_py)\tilde{q}+s_q(r_px+y\hat{p}),r_q(x\tilde{p}+s_py)+(r_px+y\hat{p})\hat{q}\big)-\big(x\widetilde{(pq)}+s_{pq}y,r_{pq}x+y\widehat{(pq)}\big)\\
		 &=\big((x\tilde{p}+s_py)\tilde{q}+s_q(r_px+y\hat{p})-x\widetilde{(pq)}-s_{pq}y,r_q(x\tilde{p}+s_py)+(r_px+y\hat{p})\hat{q}-r_{pq}x-y\widehat{(pq)}\big).
		\end{align*}
		Let $x=0$ and $y=0$ respectively, then we have:
		\begin{numcases}{}
		(x\tilde{p})\tilde{q}+s_qr_px-x\widetilde{(pq)}=[p,q,x]\label{eq:proof 2.1}\\
		r_q\tilde{p}+r_p\hat{q}-r_{pq}=0\label{eq:proof 2.2}\\
		s_p\tilde{q}+s_q\hat{p}-s_{pq}=0\label{eq:proof 2.3}\\
		r_qs_py+(y\hat{p})\hat{q}-y\widehat{(pq)}=[p,q,y]\label{eq:proof 2.4}
		\end{numcases}
	%	\item  Claim: $r_p,s_p=0,\forall p\in\spo$.
\begin{step }
	Claim: $r_p=s_p=0$, for all $p\in\spo$.
\end{step }		
		If there exists $p_0\in\spo$, such that $r_{p_0}\neq 0$, then by equation \eqref{eq:proof 2.2}, we obtain:
		$$\hat{q}=r_{p_0}^{-1}(r_{p_0q}-r_q\tilde{p_0}),\quad \forall q\in \spo.$$
		Let $\tilde{p_0}\in \spc_J$ for some imaginary unit $J\in \S$,  we conclude $\hat{q}\in \spc_J$ for all $q$.	 
%		We claim $r_q\neq0$ for all $q\in \spo$. If not, by  equation \eqref{eq:proof 2.2}, we conclude  $\exists q, s.t. \mathit{Im}\;(\hat{q})=0$	
		Let $y=J$ in equation \eqref{eq:proof 2.4}, we thus get $$[p,q,J]\in \spc_J, \text{ for all } p,q\in \spo.$$
		However this is impossible. Indeed, we can choose $p\in \spo$ orthogonal to $J$, and then choose  $q$ orthogonal to $p$ and $J$, then $[p,q,J]\notin \spc_J$. This forces that $r_p=0$ for all $p\in\spo$. We can prove $s_p=0$ for all $p\in \spo$ in the same way.
		
	%\item  Claim: $\sigma=\tau=id$.
\begin{step }
	Define $\sigma:p\mapsto \tilde{p}$ and $\tau:p\mapsto \hat{p}$.
	Claim: $\sigma=\tau=id$.
\end{step }	
		Let $x=1$ in equation \eqref{eq:proof 2.1}, we obtain $\sigma(pq)=\sigma(p)\sigma(q)$ and hence $\sigma\in G_2$. 
%		Note that $\sigma(1)=1$, it follows from  Lemma \ref{lem:ass is pure im} that $\sigma$ is an automorphism of $\spo$. 
		Suppose that $\sigma\neq id$, that is,  there exists $p\in {\spo}$, such that $\sigma(p)=\tilde{p}\neq p$. Note that $\sigma(\pureim{\spo})\subseteq\pureim{\spo}$,  which yields $\re \sigma(p)=\sigma(\re p)=\re p$, we can assume $\re p=0$.  Let $x=\tilde{p}$ in equqtion \eqref{eq:proof 2.1}, we obtain 
		$$0=[p,q,\tilde{p}], \quad \forall q\in \spo.$$ However, $\sigma$ is an automorphism of $\spo$, we can certainly choose $q\in \spo$ such that $[p,q,\tilde{p}]\neq 0$, a contradiction. Similar argument apply to $\tau$.
		
%	\item Claim: Both $\sigma:p\mapsto \tilde{p}$ and $\tau:p\mapsto \hat{p}$ are automorphism of $\spo$.
%	
%	Let $x=1$ in equation \eqref{eq:proof 2.1}, we obtain $\sigma(pq)=\sigma(p)\sigma(q)$. We claim $\sigma$ is injective. If not, then we can take $p\neq q$, such that $\sigma(p)=\sigma(q)$. According to equation \eqref{eq:proof 2.1}, we have
%	$$[p,q,x]=0,\quad \forall x\in \spo.$$
%	This implies $$\mathit{Im}\, p=r\mathit{Im}\,q \quad r\in \spr.$$
%	Clearly, $\sigma(1)=1$, $\sigma(\pureim{\spo})\subseteq\pureim{\spo}$,  this yields $\re \sigma(p)=\sigma(\re p)=\re p$. \bfs\ $\re p=\re q=0$, thus we conclude $p=rq, r\in \spr$. But this implies $$\sigma(p)=\sigma(rq)=r\sigma(q)=r\sigma(p),$$
%	which yields $r=1$ and hence $p=q$. This contradicts our assumption. Since $\sigma$ is a real linear map of finite dimensitional  vector space $\spo $, we thus conclude that $\sigma$ is an automorphism. Similar argument for $\tau$.
%	
%	\item  $\sigma=\tau=id$.
%	
%	If there exists $p\in \pureim{\spo}$, such that $\tilde{p}\neq p$, then let $x=\tilde{p}$ in equqtion \eqref{eq:proof 2.1}, we get 
%	$$0=[p,q,\tilde{p}], \quad \forall q\in \spo.$$ However $\sigma$ is automorphism of $\spo$, we can choose $q\in \spo$ such that $[p,q,\tilde{p}]\neq 0$, a contradiction.
%\end{enumerate}
In summary, the right multiplication  is just given by $(x,y)\cdot p=(xp,yp)$.
Therefore,  $\spo^2$ admits a unique compatible bimodule structure.
\end{proof}

\

%\subsection{Bimodule structure on   $\spo\oplus\overline{\spo}$}
Next we consider the case of $\spo\oplus\overline{\spo}$. 
\begin{thm}\label{thm:bimod on O+O-}
	Let the left $\spo$-module structure on $\spo\oplus\overline{\spo}$  is as follows:
	$$p(x,y)=(px,\overline{p}y).$$
	Then $\spo\oplus\overline{\spo}$ admits no compatible bimodule structures.
	% where the left multiplication on $\spo\oplus\overline{\spo}$ is defined by 
	
\end{thm}
The proof of Theorem \ref{thm:bimod on O+O-} will rely on the following two lemmas, which are also important in the sequel.
\begin{lemma}\label{lem:f=0 f(px)=overline(p)x}
	Let $f\in End_\spr(\spo)$ satisfy $f(px)=\overline{p}f(x)$ for all $p,x \in \spo$, then $f=0$.
\end{lemma}
\begin{proof}
	Let $f(1)=x_0+\sum x_ie_i$, where $x_j\in \spr, j=0,1,\ldots,7$.
	Fix $i\neq j$, $i,j\in \{1,\ldots,7\}$. We compute:
	\begin{align*}
	f(e_ie_j)&=f(\epsilon_{ijk}e_k-\delta_{ij})\\
	&=-\epsilon_{ijk}e_kf(1)\\
	&=-\epsilon_{ijk}e_kx_0- \epsilon_{ijk}x_m(\epsilon_{kmn}e_n-\delta_{km})\\
	&=-\epsilon_{ijk}e_kx_0- \epsilon_{ijk}\epsilon_{kmn}x_me_n+ \epsilon_{ijk}x_k,
	\end{align*}
	and 
	\begin{align*}
	f(e_ie_j)&=\overline{e_i}f(e_j)\\
	&=\overline{e_i}(\overline{e_j}f(1))\\
	&=e_i(e_jx_0+e_je_mx_m)\\
	&=\epsilon_{ijk}e_kx_0+e_ix_m(\epsilon_{jmn}e_n-\delta_{jm})\\
	&=\epsilon_{ijk}e_kx_0+x_m\epsilon_{jmn}(\epsilon_{inl}e_l-\delta_{in})-e_ix_j\\
	&=\epsilon_{ijk}e_kx_0+x_m\epsilon_{jmn}\epsilon_{inl}e_l-x_m\epsilon_{jmi}-e_ix_j.
	\end{align*}
	Taking the real part of both equalities infers that:
%	It follows by taking the real part of both sides that 
	$$\epsilon_{ijk}x_k=-x_m\epsilon_{jmi}=-\epsilon_{ijk}x_k.$$
	This yields $x_k=0$, where $k$ is determined by $i,j$ uniquely. Since $i,j$ are fixed arbitrarily, we conclude $f(1)=x_0\in \spr$.
	Hence $$\overline{p}(\overline{x}x_0)=\overline{p}f(x)=f(px)=\overline{px}x_0=(\overline{x}\; \overline{p})x_0.$$
	That is 
	$$x_0[\overline{x},\overline{p}]=0, \quad \forall x,p\in \spo.$$
	This leads to $x_0=0$ and hence $f=0$.
	 
\end{proof}

\begin{lemma}\label{lem:simple bimod is O}
	The left  module	$\overline{\spo}$  admits no compatible bimodule structures.
\end{lemma}
\begin{proof}
	Suppose there exits an $\spo$-bimodule structure on 	$\overline{\spo}$ with a right scalar multiplication defined  by an $\R$-linear map $R\in \End_{\R}(\O)$. Write $R_p(x)=x\tilde{\cdot}p$. 
%	Viewed as one-sided $\spo$-module, we deduce that $f_p$ is  
	By the definition of $\O$-bimodule, we have 
	$$[p,q,x]_{\overline{\O}}=[q,x,p]_{\overline{\O}},\quad \forall p,q,x\in \O.$$
	Note the equation \eqref{eq:[]O-}, we obtain:
	$$[p,q,x]+[\conjgt{q},\conjgt{p}]x=R_p(\conjgt{q}x)-\conjgt{q}{R_p(x)}.$$
	Replacing $\conjgt{q}$ with $q$, it becomes:
	\begin{equation}\label{eq:pf Rp(qx) O- no bimod}
	R_p(qx)=qR_p(x)-[p,q,x]+[q,\conjgt{p}]x
	\end{equation}
	Let $x=1$ in \eqref{eq:pf Rp(qx) O- no bimod}, we get:
	\begin{equation}\label{eq:pf Rp(q) O- no bimod}
	R_p(q)=qR_p(1)+[q,\conjgt{p}]
	\end{equation}
	It follows that 
	\begin{align*}
	R_p(qx)&=qR_p(x)-[p,q,x]+[q,\conjgt{p}]x\\
	&=q(xR_p(1)+[x,\conjgt{p}])-[p,q,x]+[q,\conjgt{p}]x
	\intertext{and }
	R_p(qx)&=(qx)R_p(1)+[qx,\conjgt{p}].
	\end{align*}
	%$$	R_p(qx)=(qx)R_p(1)+[qx,\conjgt{p}].$$
	Hence we conclude 
	\begin{align*}
	0&=(qx)R_p(1)+[qx,\conjgt{p}]-\Big(q(xR_p(1)+[x,\conjgt{p}])-[p,q,x]+[q,\conjgt{p}]x\Big)\\
	&=[q,x,R_p(1)]+(qx)\conjgt{p}-\conjgt{p}(qx)-q(x\conjgt{p}-\conjgt{p}x)-(q\conjgt{p}-\conjgt{p}q)x-[\conjgt{p},q,x]\\
	&=[R_p(1),q,x]+2[\conjgt{p},q,x]\\
	&=[R_p(1)-2{p},q,x].
	\end{align*}
	Since the above  equation holds for any $p,q,x\in \O$, this yields $$R_p(1)-2{p}\in \R, \quad \forall p\in \O.$$
	Hence we can assume $R_{e_1}=2e_1+r$ for some $r\in \R$. Note that formula \eqref{eq:pf Rp(q) O- no bimod} ensures $R_p(p)=pR_p(1)$, it follows that
	\begin{align*}
	R_{e_1}(R_{e_1}1)&=R_{e_1}(2e_1+r)\\
	&=2e_1(e_1+r)+r(2e_1+r)\\
	&=-4+4re_1+r^2.
	\end{align*}
	However, $R_{e_1}(R_{e_1}1)=R_{e_1^2}1=-1$, and hence we obtain
	$$-1=-4+4re_1+r^2,$$
	for some $r\in \R$, this is impossible.  This proves the lemma.	
\end{proof}
%\begin{thm}
%	Let the left $\spo$-module structure of $\spo\oplus\overline{\spo}$  is 
%	$$p(x,y)=(px,\overline{p}y).$$
%	Then $\spo\oplus\overline{\spo}$ admits no bimodule structures.
%	% where the left multiplication on $\spo\oplus\overline{\spo}$ is defined by 
%	
%\end{thm}
\begin{proof}[Proof of Theorem \ref{thm:bimod on O+O-}.]
	Suppose $\spo\oplus\overline{\spo}$ admits a compatible bimodule structure, and  the  right multiplication is as follows:
	$$(x,y)\cdot p=\big(f_p(x)+h_p(y),g_p(x)+l_p(y)\big).$$
	Similar as before, we can derive that $f_p,g_p,h_p,l_p\in End_\spr(\spo)$ for all $p\in \spo$, and are all real linear on $p$.
	Let $[p,q,(x,y)]=[q,(x,y),p]$, we obtain
	\begin{numcases}{}
	f_p(qx)-qf_p(x)=[p,q,x]\label{eq:proof 3.1}\\
	g_p(qx)-\overline{q}g_p(x)=0\label{eq:proof 3.2}\\
	h_p(\overline{q}y)-qh_p(y)=0\label{eq:proof 3.3}\\
	l_p(\overline{q}y)-\overline{q}l_p(y)=[p,q,y]_{\overline{\O}}\label{eq:proof 3.4}
	\end{numcases}
	Hence by Lemma \ref{lem:f=0 f(px)=overline(p)x}, $g_p=h_p=0$ for all $p\in \spo$.
	Let $[p,q,(x,y)]=[(x,y),p,q]$, we obtain
	\begin{gather}
	l_q(l_p(y))-l_{pq}(y)=[p,q,y]_{\overline{\O}}\label{eq:proof 3.5}
	\end{gather}
	The fact that $l_p(x)$ is real linear on $p$ and $x$,  along with the equation \eqref{eq:proof 3.5} imply that $$y\cdot_l p:=l_p(y)$$ defines a right $\spo$-module structure on $\overline{\spo}$ and satisfies $[p,q,y]_{\overline{\O}}=[y,p,q]_{\overline{\O}}$. Note that equation \eqref{eq:proof 3.4} yields
	$[p,q,y]_{\overline{\O}}=[q,y,p]_{\overline{\O}}$, this means that it  defines an $\spo$-bimodule structure on $\overline{\spo}$, which contradicts  the Lemma \ref{lem:simple bimod is O}.
\end{proof}
\subsection{Bimodule structures on finite dimensional $\O$-modules}
In this subsection, we will formulate the structure of finite dimensional $\O$-bimodules. As is shown in \cite{liyong2019octonionmodule}, each finite dimensional left $\O$-module $M$ is of the form: $$	M\cong \O^{ n}\oplus \overline{\spo}^{m}.$$
   $\spo^n$ is a left $\O$-module endowed with the left multiplication: $$p(x_1,\ldots,x_n)=(px_1,\ldots,px_n).$$  $\overline{\spo}^{n}\oplus\spo^m$ is a left $\O$-module endowed with the left multiplication:
$$p(x_1,\ldots,x_n,x_{n+1},\ldots,x_{n+m})=(\overline{p}x_1,\ldots,\overline{p}x_n,px_{n+1},\ldots,px_{n+m}).$$

 We are first concerned  with the case $\O^n$. The following theorem asserts  that it  admits a unique compatible bimodule structure.
\begin{thm}\label{thm:bimod on O^n}
	
	There exists a unique compatible bimodule structure on $\spo^n$.
	
	%$\spo$-bimod is tensor product., what's more, the structure is   tensor product
\end{thm}
\begin{proof}
	
	Suppose there exists a compatible  bimodule structure on $\spo^n$ and the right multiplication is as follows: 
	$$(0,\ldots,0,x_i,0,\ldots,0)\cdot p=(f_{i1}(p;x_i),\ldots,f_{in}(p;x_i)).$$
	Then $$(x_1,\ldots,x_n)\cdot p=\big(\sum f_{i1}(p;x_i),\ldots,\sum f_{in}(p;x_i) \big).$$
	Similar as before,  $f_p,g_p,h_p,l_p\in End_\spr(\spo)$ are real linear maps  for all $p\in \spo$  and also   real linear on $p$.
\begin{step }
	 $[p,q,(x_1,\ldots,x_n)]=[q,(x_1,\ldots,x_n),p]$.
\end{step }	
	
	Let  $[p,q,(x_1,\ldots,x_n)]=[q,(x_1,\ldots,x_n),p]$, we obtain	
$$	[p,q,x_j]=\sum f_{ij}(p;qx_i)-qf_{ij}(p;x_i),\quad j=1,\ldots,n.$$
Let $i_0\in \{1,\ldots,n\}$, set  $$x_i=\begin{cases}
x, &i=i_0\\
0, &i\neq i_0
\end{cases},$$
then we obtain:
\begin{numcases}{}
f_{i_0j}(p;qx)-qf_{i_0j}(p;x)=0, &$\quad j\neq i_0$\label{eq:proof 4.1}\\
f_{i_0i_0}(p;qx)-qf_{i_0i_0}(p;x)=[p,q,x],&$i_0\in \{1,\ldots,n\}$\label{eq:proof 4.2}
\end{numcases}
%\begin{gather}
%f_{i_0j}(p;qx)-qf_{i_0j}(p;x)=0, \quad j\neq i_0\label{eq:proof 4.1}\\
%f_{i_0i_0}(p;qx)-qf_{i_0i_0}(p;x)=[p,q,x]\label{eq:proof 4.2}
%\end{gather}
Since $i_0$ is fixed arbitrarily, we conclude from  equations \eqref{eq:proof 4.1} that $f_{ij}(p;x)$ is $\spo$-homomorphism when $i\neq j$. hence we can assume as before
$$f_{ij}(p;x)=r_{ij}(p)x, \quad r_{ij}(p)\in \spr, \; i\neq j.$$
Equations \eqref{eq:proof 4.2} enable us to assume
$$f_{ii}(p;x)=xr_{ii}(p), \quad r_{ii}(p)\in \spo.$$

\begin{step }
	$[p,q,(x_1,\ldots,x_n)]=[(x_1,\ldots,x_n),p,q]$.
\end{step }
Let $[p,q,(x_1,\ldots,x_n)]=[(x_1,\ldots,x_n),p,q]$, we obtain
$$[p,q,x_l]=\sum \big(x_ir_{ik}(p)\big)r_{kl}(q)-\sum x_ir_{il}(pq).$$ 
Let  $$x_i=\begin{cases}
x, &i=l_0\\
0, &i\neq l_0
\end{cases},$$
then we obtain:
\begin{numcases}{}
[p,q,x]=\sum \big(xr_{l_0k}(p)\big)r_{kl_0}(q)- xr_{l_0l_0}(pq),&$l_0\in \{1,\ldots,n\}$\label{eq:proof 4.3}\\
0=\sum \big(xr_{l_0k}(p)\big)r_{kl}(q)- xr_{l_0l}(pq),&$\quad l\neq l_0$\label{eq:proof 4.4}
\end{numcases}
%\begin{gather}
%[p,q,x]=\sum \big(xr_{l_0k}(p)\big)r_{kl_0}(q)- xr_{l_0l_0}(pq)\label{eq:proof 4.3}\\
%0=\sum \big(xr_{l_0k}(p)\big)r_{kl}(q)- xr_{l_0l}(pq)\quad l\neq l_0\label{eq:proof 4.4}
%\end{gather} 
Note that $r_{ij}(p)\in \spr$ for any distinct indices $i$ and $j$, hence equations \eqref{eq:proof 4.4} are equivalent to 
\begin{eqnarray}
\sum r_{l_0k}(p)r_{kl}(q)- r_{l_0l}(pq)=0,\quad l\neq l_0\label{eq:proof 4.5}
\end{eqnarray}
\begin{step }
	Claim: $r_{jl}(p)=0$ for all $p\in \spo,\; j\neq l$.
\end{step }
%Claim: $r_{jl}(p)=0$ for all $p\in \spo,\; j\neq l$.

Suppose on the contrary, there exists $p\in \spo$, and $j_0\neq i_0$, such that $r_{i_0j_0}(p)\neq 0$. Let $l=j_0$ and take imaginary part on both sides of equations \eqref{eq:proof 4.5}, we obtain
$$\imaginary \big(r_{i_0i_0}(p)r_{i_0j_0}(q)+r_{i_0j_0}(p)r_{j_0j_0}(q)\big)=0,$$
thus $$\imaginary r_{j_0j_0}(q)=-r_{i_0j_0}(p)^{-1}r_{i_0j_0}(q)\imaginary r_{i_0i_0}(p).$$
Suppose $r_{i_0i_0}(p)\in \spc_J$ for some imaginary unit $J$, then we conclude that 
$$r_{j_0j_0}(q)\in \spc_J, \quad \forall q\in \spo.$$
Replacing $l_0$ with $j_0$ and $x$ with $J$ in equations \eqref{eq:proof 4.3}, we get
$$[p,q,J]\in \spc_J,\quad \forall p,q\in \spo.$$
Thus we have arrived at a contradiction.

Now equations \eqref{eq:proof 4.3} become
$$[p,q,x]=\big(xr_{ll}(p)\big)r_{ll}(q)- xr_{ll}(pq), \quad l=1,\ldots,n.$$
As in the proof of Theorem \ref{thm:O2 bimod} of the case $n=2$, we can deduce $r_{ll}=id$ for $l=1,\ldots,n$. This completes the proof.
%shows the bimodule structure on $\spo^n$ is uniquely determined by its left module. 
\end{proof}

By  similar argument, we can prove:
\begin{thm}\label{thm:bimod on On+O-m}
 There exist no bimodule structures on $\overline{\spo}^{n}\oplus\spo^m$ when $n>0$.
%  where the left multiplication is 
% $$p(x_1,\ldots,x_n,x_{n+1},\ldots,x_{n+m})=(\overline{p}x_1,\ldots,\overline{p}x_n,px_{n+1},\ldots,px_{n+m}).$$
\end{thm}
\begin{proof}
	Suppose on the contrary there exists an $\spo$-bimodule structure on  $\overline{\spo}^{n}\oplus\spo^m$ and the right scalar multiplication is given by:
	$$(x_1,\ldots,x_n,x_{n+1},\ldots,x_{n+m})\cdot p=\left(\sum_{i=1}^N  f_{i1}(p;x_i),\ldots,\sum_{i=1}^N f_{in}(p;x_i)\right),$$
	where $N=n+m$.
	
	\begin{step }
	$[p,q,(x_1,\ldots,x_n,x_{n+1},\ldots,x_{n+m})]=[q,(x_1,\ldots,x_n,x_{n+1},\ldots,x_{n+m}),p]$.% for any $m=(x_1,\ldots,x_N)\in \overline{\spo}^{n}\oplus\spo^m $.
\end{step }
We first compute $[q,(x_1,\ldots,x_N),p]$.
\begin{align*}
[q,(x_1,\ldots,x_N),p]&=(\overline{q}x_1,\ldots,\overline{q}x_n,qx_{n+1},\ldots,qx_{n+m})\cdot p-q\left(\sum_{i=1}^N  f_{i1}(p;x_i),\ldots,\sum_{i=1}^N f_{in}(p;x_i)\right)\\
&=\left(\sum_{i=1}^n f_{ij}(p;\overline{q}x_i)-\overline{q}f_{ij}(p;x_i)+\sum_{i=n+1}^N f_{ij}(p;qx_i)-\overline{q}f_{ij}(p;x_i)\right)_{j=1}^n+\\
&\quad\left(\sum_{i=1}^n f_{ij}(p;\overline{q}x_i)-{q}f_{ij}(p;x_i)+\sum_{i=n+1}^N f_{ij}(p;qx_i)-{q}f_{ij}(p;x_i)\right)_{j=n+1}^N,
\end{align*}
where $(x_j)_{j=1}^n:=(x_1,\ldots,x_n,0,\ldots,0)\in \overline{\spo}^{n}\oplus\spo^m$, similar for $(x_j)_{j=n+1}^N$.

By the  definition of $\O$-bimodule, we have $$[p,q,(x_1,\ldots,x_n,x_{n+1},\ldots,x_{n+m})]=[q,(x_1,\ldots,x_n,x_{n+1},\ldots,x_{n+m}),p].$$ %$[q,p,m]=[p,m,q]$ for any $m=(x_1,\ldots,x_N)\in \overline{\spo}^{n}\oplus\spo^m $, 
Fix $j_0\in \{1,\ldots,N\}$ and let 
$$x_i=
\begin{cases}
x, &i=j_0\\
0, &i\neq j_0
\end{cases}.$$ 
If $j_0 \in \{1,\ldots,n\}$, we obtain:
%\begin{align}
%f_{j_0j_0}(p;\overline{q}x)-\overline{q}f_{j_0j_0}(p;x)&=[p,q,x]_{\overline{\O}}\label{eq:proof 5.1}\\
%f_{j_0j}(p;\overline{q}x)-\overline{q}f_{j_0j}(p;x)&=0 \quad j_0\neq j\in \{1,\ldots,n\}\label{eq:proof 5.2}\\
%f_{j_0j}(p;\overline{q}x)-{q}f_{j_0j}(px)&=0 \quad j\in \{n+1,\ldots,N\}\label{eq:proof 5.3}
%\end{align}
\begin{numcases}{}
f_{j_0j_0}(p;\overline{q}x)-\overline{q}f_{j_0j_0}(p;x)=[p,q,x]_{\overline{\O}}, &$j_0\in\{1,\ldots,n\}$\label{eq:proof 5.1}\\
f_{j_0j}(p;\overline{q}x)-\overline{q}f_{j_0j}(p;x)=0, &$j_0\neq j\in \{1,\ldots,n\}$\label{eq:proof 5.2}\\
f_{j_0j}(p;\overline{q}x)-{q}f_{j_0j}(px)=0, &$j\in \{n+1,\ldots,N\}$\label{eq:proof 5.3}
\end{numcases}
If $j_0 \in \{n+1,\ldots,N\}$, we obtain:
\begin{numcases}{}
f_{j_0j_0}(p;{q}x)-{q}f_{j_0j_0}(p;x)=[p,q,x], &$j_0 \in \{n+1,\ldots,N\}$\label{eq:proof 5.4}\\
f_{j_0j}(p;{q}x)-\overline{q}f_{j_0j}(p;x)=0, &$  j\in \{1,\ldots,n\}$\label{eq:proof 5.5}\\
f_{j_0j}(p;{q}x)-{q}f_{j_0j}(p;x)=0,&$ j_0\neq j\in \{n+1,\ldots,N\}$\label{eq:proof 5.6}
\end{numcases}
%\begin{align}
%f_{j_0j_0}(p;{q}x)-{q}f_{j_0j_0}(p;x)&=[p,q,x]\label{eq:proof 5.4}\\
%f_{j_0j}(p;{q}x)-\overline{q}f_{j_0j}(p;x)&=0 \quad  j\in \{1,\ldots,n\}\label{eq:proof 5.5}\\
%f_{j_0j}(p;{q}x)-{q}f_{j_0j}(p;x)&=0 \quad j_0\neq j\in \{n+1,\ldots,N\}\label{eq:proof 5.6}
%\end{align}
By Lemma \ref{lem:f=0 f(px)=overline(p)x} and equations \eqref{eq:proof 5.3} and \eqref{eq:proof 5.5}, we conclude that $f_{ij}=0$ for $i\in \{1,\ldots,n\}$, $j\in \{n+1,\ldots,N\}$ or $j\in \{1,\ldots,n\}$, $i\in \{n+1,\ldots,N\}$; the same as before, we can assume 
$$f_{j_0j}(p;x)=r_{j_0j}(p)x,\quad r_{j_0j}(p)\in\spr$$ for $j_0,j\in \{1,\ldots,n\}, j_0\neq j$ and $j_0,j\in \{n+1,\ldots,N\}, j_0\neq j$.

\begin{step }
	$[p,q,(x_1,\ldots,x_n,x_{n+1},\ldots,x_{n+m})]=[(x_1,\ldots,x_n,x_{n+1},\ldots,x_{n+m}),p,q]$.% for any $m=(x_1,\ldots,x_N)\in \overline{\spo}^{n}\oplus\spo^m $.
\end{step }
Now we compute $[(x_1,\ldots,x_N),p,q]$.
\begin{align*}
[(x_1,\ldots,x_N),p,q]&=\left(\sum _{i=1}^N f_{ij}(p;x_i)\right)_{j=1}^N\cdot q-\left(\sum _{i=1}^N f_{ij}(pq;x_i)\right)_{j=1}^N\\
&=\left(\sum _{k=1}^N f_{kj}\Big(q;\sum _{i=1}^N f_{ik}(p;x_i)\Big)-\sum _{i=1}^N f_{ij}(pq;x_i)\right)_{j=1}^N.
\end{align*}
Similar as before, we have:
\begin{numcases}{}
\sum _{k=1}^N f_{kj_0}\big(q; f_{j_0k}(p;x)\big)-f_{j_0j_0}(pq;x)=[p,q,x]_{\overline{\O}},&$\quad j_0\in \{1,\ldots,n\}$ \label{eq:prf 7}\\
%\sum _{k=1}^N f_{kj}\big(q; f_{j_0k}(p;x)\big)-f_{j_0j}(pq;x)&=0&\quad j\neq j_0,\; j_0\in \{1,\ldots,n\} \\
\sum _{k=1}^N f_{kj_0}\big(q; f_{j_0k}(p;x)\big)-f_{j_0j_0}(pq;x)=[p,q,x],&$\quad j_0\in \{n+1,\ldots,N\}$\\
\sum _{k=1}^N f_{kj}\big(q; f_{j_0k}(p;x)\big)-f_{j_0j}(pq;x)=0,&$\quad j\neq j_0,\;  j,j_0\in \{1,\ldots,N\}$ \label{eq:prf 8}
\end{numcases}
%\begin{align}
%\sum _{k=1}^N f_{kj_0}\big(q; f_{j_0k}(p;x)\big)-f_{j_0j_0}(pq;x)&=[p,q,x]_{\overline{\O}}&\quad j_0\in \{1,\ldots,n\} \label{eq:prf 7}\\
%%\sum _{k=1}^N f_{kj}\big(q; f_{j_0k}(p;x)\big)-f_{j_0j}(pq;x)&=0&\quad j\neq j_0,\; j_0\in \{1,\ldots,n\} \\
%\sum _{k=1}^N f_{kj_0}\big(q; f_{j_0k}(p;x)\big)-f_{j_0j_0}(pq;x)&=[p,q,x]&\quad j_0\in \{n+1,\ldots,N\}\\
%\sum _{k=1}^N f_{kj}\big(q; f_{j_0k}(p;x)\big)-f_{j_0j}(pq;x)&=0&\quad j\neq j_0,\;  j,j_0\in \{1,\ldots,N\} \label{eq:prf 8}
%\end{align}

Note what we have just proved, we can rewrite equations \eqref{eq:prf 7} as follows:
\begin{eqnarray}\label{eq:prof 7'}
\sum _{k=1,k\neq j_0}^n r_{kj_0}(q)r_{j_0k}(p)x+f_{j_0j_0}\big(q; f_{j_0j_0}(p;x)\big)-f_{j_0j_0}(pq;x)=[p,q,x]_{\overline{\O}} \quad j_0\in \{1,\ldots,n\}
\end{eqnarray}
Rewrite equations \eqref{eq:prf 8} as follows:
\begin{eqnarray}\label{eq:prof 8'}
\sum _{k=1,k\neq j_0,k\neq j}^n r_{kj_0}(q)r_{j_0k}(p)x+f_{jj}(q;r_{j_0j}(p)x)+r_{j_0j}(q)f_{j_0j_0}(p;x)-r_{j_0j}(pq)x=0
\end{eqnarray}
The equations \eqref{eq:prof 8'} hold for $ j\neq j_0,\;  j,j_0\in \{1,\ldots,N\}$. %$j_0,j\in \{1,\ldots,n\}$ and $j_0\neq j$.

\begin{step }
	 $r_{ij}(p)=0$ for all $i\neq j$ in $\{1,\ldots,n\}$.
\end{step }

Taking imaginary part on both sides of equations \eqref{eq:prof 8'}, we get 
$$\imaginary \Big(r_{j_0j}(p)f_{jj}(q;x)+r_{j_0j}(q)f_{j_0j_0}(p;x)\Big)=0.$$
If there exists an octonion $ p\in \spo$, and $ l_0,l\in \{1,\ldots,n\}, l_0\neq l$ such that $r_{l_0l}(p)\neq 0$,  let $x=1$ and   $f_{l_0l_0}(p;1)\in \spc_J$,  we then conclude as before
$$f_{ll}(q;1)\in \spc_J,\quad \forall q\in \spo.$$
However, repalcing $j_0$ by $l$ in equations \eqref{eq:prof 7'}, we conclude $$[p,q,1]_{\overline{\O}}=\overline{[p,q]}\in \spc_J, \quad \forall p,q\in \spo.$$ Obviously this is impossible. Thus we have arrived at a contradiction. This shows $r_{ij}(p)=0$ for all $i\neq j$ in $\{1,\ldots,n\}$.

\begin{step }
For each $j=1,\ldots,n$, $f_{jj}(p;x)$ 	defines a bimodule structure on $\overline{\O}$.
\end{step }
Thanks to Step 3, equations \eqref{eq:prof 7'} become 
$$f_{j_0j_0}\big(q; f_{j_0j_0}(p;x)\big)-f_{j_0j_0}(pq;x)=[p,q,x]_{\overline{\O}}, \quad j_0\in \{1,\ldots,n\}.$$
This imply that we get a right  $\spo$-module sturcture on $\overline{\spo}$ with the right multiplication defined by $x \cdot_j p:=f_{jj}(p;x)$ for each $j=1,\ldots,n$. Moreover, combining equations \eqref{eq:proof 5.1} and  \eqref{eq:prof 7'} yields an $\spo$-bimodule structure  on $\overline{\spo}$, which contradicts the Lemma \ref{lem:simple bimod is O}. This proves the theorem.
\end{proof}

\subsection{The structure of general $\O$-bimodule}
In this subsection, we are in a position to deal with the bimodule structure of general left $\O$-modules.
We have shown that each left $\O$-module has a ``basis'' in a previous paper \cite{liyong2019octonionmodule}, this loosely means that each left $\O$-module is a “free” module.   
In   much the same way as finite dimensional case, we can prove that a left $\O$-module $M$ admits a compatible bimodule structure if and only if  $M=\O \huaa{M}$. Moreover, the bimodule structure is unique if it exists. 

In view of identity \eqref{eq:[p,q,r]m+p[q,r,m]=[pq,r,m]-[p,qr,m]+[p,q,rm]}, it holds $[p,q,rx]=[p,q,r]x$ for any  associative element $x\in \huaa{M}$.  We now give a similar formula for conjugate associative element.
For conveniention, we define a new \ass, denoted by  $\dbb{p}{q}{r}:=[p,q,r]+r[p,q]$. Then by direct calculation,  we have for any conjugate associative element $x\in \hua{A}{-}{M}$:
%$$[p,q,rx]=\dbb{p}{q}{r}x.$$
\begin{eqnarray}
[p,q,rx]=\dbb{p}{q}{r}x.
\end{eqnarray}
In fact, let $x\in \hua{A}{-}{M}$,
\begin{align*}
[p,q,rx]&=(pq)(rx)-p(q(rx))\\
&=(r(pq)-(rq)p)x\\
&=(r[p,q]-[r,q,p])x\\
&=\dbb{p}{q}{r}x.
\end{align*} 
By the way, we can give an alternative  derivation of  the \ass\ of $\overline{\O}$  as follows:
\begin{align*}
[p,q,x]_{\overline{\O}}&=[p,q,\overline{x}\hat{\cdot}1]_{\overline{\O}}
=\dbb{p}{q}{\overline{x}}\hat{\cdot}1
=\overline{\dbb{p}{q}{\overline{x}}}
=[p,q,x]+\overline{[p,q]}x.
\end{align*}
In particular, we get 
\begin{eqnarray}\label{eq:[p,q,x]O-=[p,q,x]-}
[p,q,x]_{\overline{\O}}=\overline{\dbb{p}{q}{\overline{x}}}.
\end{eqnarray}
%$$[p,q,x]_{\overline{\O}}=\overline{\dbb{p}{q}{\overline{x}}}.$$

%From above we can see, $$[p,q,x]_{\overline{\O}}=\overline{\dbb{p}{q}{\overline{x}}},$$
%that is, $\dbb{p}{q}{x}=\overline{[p,q,\overline{x}]_{\overline{\O}}}$.

\begin{thm}\label{thm:general bimod strc}
	A left $\O$-module $M$ admits a compatible bimodule structure if and only if  $M=\O \huaa{M}$. 
	
	Moreover, in  this case, the right scalar multiplication on  $  \huaa{M}$ coincides with the left scalar multiplication: 
	$$xp=px,\quad \forall p\in \O,\;\forall x\in \huaa{M}.$$
	And   this determines the right scalar multiplication on $M$.
\end{thm}

We first prove a simple lemma  which will be used later.
\begin{lemma}\label{lem:f=0 f(xq)=qf(x)}
	Let $f\in \End_{\R}(\O)$. If  it holds $f(xq)=qf(x)$ for all $q,x\in \O$, then $f=0$.
\end{lemma}
\begin{proof}
	This is a simple deformation of Lemma \ref{lem:f=0 f(px)=overline(p)x}. We define $g(x):=f(\overline{x})$, then we obtain:
	$$g(px)=f(\overline{x}\overline{p})=\overline{p}f(\overline{x})=\overline{p}g({x}).$$ It thus follows from Lemma \ref{lem:f=0 f(px)=overline(p)x} that $g=0$ and hence $f=0$.
\end{proof}
\begin{proof}[Proof of Theorem \ref{thm:general bimod strc}]
	Suppose $M\cong(\oplus_{i\in \Lambda_1}\spo) \bigoplus (\oplus_{\alpha\in \Lambda_2}\overline{\spo})$. Hence there is a canonical basis $\{\epsilon_i,\epsilon_{\alpha}\}_{i\in \Lambda_1,\alpha\in \Lambda_2}$, such that $\epsilon_i\in \huaa{M}$ and $\epsilon_\alpha\in \hua{A}{-}{M}$ for each $i\in \Lambda_1$ and $\alpha\in \Lambda_2$. We assume there exists an $\O$-bimodule structure and for any $x\in \O$, the right multiplication is supposed to be:
	$$(x\epsilon_i)\cdot p=\sum _{j\in \Lambda_1}f_{ij}(p;x)\epsilon_j+\sum_{\beta\in \Lambda_2}f_{i\beta}(p;x)\epsilon_{\beta};$$
	$$(x\epsilon_{\alpha})\cdot p=\sum _{j\in \Lambda_1}f_{\alpha j}(p;x)\epsilon_j+\sum_{\beta\in \Lambda_2}f_{\alpha\beta}(p;x)\epsilon_{\beta}.$$
	Note that these sums here are all finite sums.
	Therefore, 
	\begin{align*}
	&\left(\sum_{i\in \Lambda_1}x_i\epsilon_i+\sum_{\alpha\in \Lambda_1}x_\alpha\epsilon_{\alpha}\right)\cdot p\\
	=&\sum _{j\in \Lambda_1}\left(\sum _{i\in \Lambda_1}f_{ij}(p;x_i)+\sum _{\alpha\in \Lambda_2}f_{\alpha j}(p;x_{\alpha})\right)\epsilon_j+\sum _{\beta\in \Lambda_2}\left(\sum _{i\in \Lambda_1}f_{i\beta}(p;x_{i})+\sum _{\alpha\in \Lambda_2}f_{\alpha \beta}(p;x_\alpha)\right)\epsilon_\beta.
	\end{align*}
	
	%$$\left(\sum_{i\in \Lambda_1}x_i\epsilon_i+\sum_{\alpha\in \Lambda_1}x_\alpha\epsilon_{\alpha}\right)\cdot p=\sum _{j\in \Lambda_1}\left(\sum _{i\in \Lambda_1}f_{ij}(p;x_i)+\sum _{\alpha\in \Lambda_2}f_{\alpha j}(p;x_{\alpha})\right)\epsilon_j+\sum _{\beta\in \Lambda_2}\left(\sum _{i\in \Lambda_1}f_{i\beta}(p;x_{i})+\sum _{\alpha\in \Lambda_2}f_{\alpha \beta}(p;x_\alpha)\right)\epsilon_j.$$
Given $m=\sum_{i\in \Lambda_1}x_i\epsilon_i+\sum_{\alpha\in \Lambda_1}x_\alpha\epsilon_{\alpha}$, we compute $[q,m,p]$.
Note that $\epsilon_i\in \huaa{M}$ and $\epsilon_\alpha\in \hua{A}{-}{M}$, which means { for all } $p,q\in \O$, it holds
$$p(q\epsilon_i)=(pq)\epsilon_i, \quad p(q\epsilon_\alpha)=(qp)\epsilon_\alpha$$ for every $i\in \Lambda_1$ and $\alpha\in \Lambda_2$.
Consequently,
\begin{align*}
[q,m,p]&=\left(q\sum_{i\in \Lambda_1}x_i\epsilon_i+q\sum_{\alpha\in \Lambda_1}x_\alpha\epsilon_{\alpha}\right)p-q\sum _{j\in \Lambda_1}\left(\sum _{i\in \Lambda_1}f_{ij}(p;x_i)+\sum _{\alpha\in \Lambda_2}f_{\alpha j}(p;x_{\alpha})\right)\epsilon_j-\\
&\quad  q\sum _{\beta\in \Lambda_2}\left(\sum _{i\in \Lambda_1}f_{i\beta}(p;x_{i})+\sum _{\alpha\in \Lambda_2}f_{\alpha \beta}(p;x_\alpha)\right)\epsilon_\beta\\
&=\sum _{j\in \Lambda_1}\Bigg[\sum_{i\in \Lambda_1}\Big(f_{ij}(p;qx_i)-qf_{ij}(p;x_i)\Big)+\sum_{\alpha\in \Lambda_2}\Big(f_{\alpha j}(p;x_\alpha q)-qf_{\alpha j}(p;x_\alpha)\Big)\Bigg]\epsilon_j+\\
&\quad \sum _{\beta\in \Lambda_2}\Bigg[\sum_{i\in \Lambda_1}\Big(f_{i\beta}(p;qx_i)-f_{i\beta}(p;x_i)q\Big)+\sum_{\alpha\in \Lambda_2}\Big(f_{\alpha \beta}(p;x_\alpha q)-f_{\alpha \beta}(p;x_\alpha)q\Big)\Bigg]\epsilon_\beta.
\end{align*}
As before, we obtain:
\begin{numcases}{}
f_{jj}(p;qx)-qf_{jj}(p;x)=[p,q,x_{j}], &$j\in \Lambda_1$\\
f_{ij'}(p;qx)-qf_{jj'}(p;x)=0, &$j\neq j',\;j,j'\in \Lambda_1$\\
f_{j\beta}(p;qx)-f_{j\beta}(p;x)q=0, &$j\in \Lambda_1, \beta\in \Lambda_2$\label{eq:general case 0}\\
f_{\beta \beta}(p;x q)-f_{\beta \beta}(p;x)q=\dbb{p}{q}{x}, &$\beta\in \Lambda_2$\label{eq:general case 1}\\
f_{\beta \beta'}(p;x q)-f_{\beta \beta'}(p;x)q=0, &$\beta\neq \beta',\; \beta,\beta'\in \Lambda_2$\\
f_{\beta j}(p;x q)-qf_{\beta j}(p;x)=0, &$j\in \Lambda_1, \beta\in \Lambda_2$\label{eq:general case 2}
\end{numcases}
where $x$ is an arbitrary octonion. 

Thanks to Lemma \ref{lem:f=0 f(xq)=qf(x)}, we deduce from the equations \eqref{eq:general case 0} and \eqref{eq:general case 2} that $f_{j\beta}=f_{\beta j}=0$ for all $j\in \Lambda_1, \beta\in \Lambda_2$ as before.
What seems slightly different from before is the equations \eqref{eq:general case 1}. However, if we define 
$$g_{\beta \beta}(p;x)=\overline{f_{\beta\beta}(p;\overline{x})},$$
we then have:
\begin{align*}
g_{\beta \beta}(p;\overline{q}x)-\overline{q}g_{\beta \beta}(p;x)&=\overline{f_{\beta\beta}(p;\overline{x}q)}-\overline{q}\overline{f_{\beta\beta}(p;\overline{x})}\\
&=\overline{f_{\beta\beta}(p;\overline{x}q)-f_{\beta\beta}(p;\overline{x})q}\\
&=\overline{\dbb{p}{q}{\overline{x}}}\\
&=[p,q,x]_{\overline{\O}}.
\end{align*}
Where we have used the equality \eqref{eq:[p,q,x]O-=[p,q,x]-} in the last line.
The rest proof runs completely in the same manner as in Theorem  \ref{thm:bimod on O^n} and Theorem \ref{thm:bimod on On+O-m}.
\end{proof}

%
%\
%
%
%\begin{cor}\label{thm:bimod structure}
%	A left $\O$-module $M$ admits bimodule structure if and only if  $M=\O \huaa{M}$. 
%	
%	Moreover, in  this case, the right multiplication is just given by 
%	$$xp:=px,\quad \forall p\in \O,\;\forall x\in \huaa{M}.$$
%\end{cor}

\subsection{Some consequences}
Let $M$ be an $\O$-bimodule and define the \textbf{communicating center} of $M$ $$\hua{Z}{}{M}:=\{x\in M\mid px=xp, \text{ for all }p\in \O\}.$$ 
Then it turns out that the communicating center is exactly   the set $\huaa{M}$.
\begin{prop}\label{lem:Z(M) in huaa(M)}
	Let $M$ be an $\O$-bimodule. Then $\huaa{M}=\hua{Z}{}{M}$. 	
	%Where $\hua{Z}{}{M}$ is the \textbf{communicating center} of $M$.

\end{prop}
%\begin{lemma}
%	Let $M$ be an $\spo$-bimodule. Then $ \hua{Z}{}{M}\subseteq \huaa{M}$.
%\end{lemma}
\begin{proof}
	Let $x\in  \hua{Z}{}{M}$, then for any $p,q\in \spo$, we have
	\begin{align*}
	[p,q,x]&=(pq)x-p(qx)\\
	&=x(pq)-p(xq)\\
	&=(xp)q-[x,p,q]-p(xq)\\
	&=(px)q-p(xq)-[p,q,x]\\
	&=-2[p,q,x].
	\end{align*}
	Thus  $[p,q,x]=0$  for any $p,q\in \spo$ and hence $x\in \huaa{M}$. On the other hand, if $x\in \huaa{M}$, we clearly have $x\in  \hua{Z}{}{M}$ in view of Theorem \ref{thm:general bimod strc}. This proves the proposition.
\end{proof}
%but we may omit the prefixes when the modulle under consideration is understood well. In the case of bimodule, no prefix is used, although when it is, for example, considered as a left module, the prefix '$l$-' is then added. For example,

In order to avoid any confusion, we use the prefixes $l$- and $r$- to indicate that the modulle under consideration is a left or  right module. For example, let $M$ and $M'$ be two $\spo$-bimodules,  we use $l$-$\Hom_\O(M,M')$ to denote the set of all left homomorphisms over  $\O$-bimodules $M$ and $M'$, similar notation  $r$-$\Hom_\O(M,M')$ for right homomorphisms, and  use $\Hom_\O(M,M')$ to denote the set of bihomomorphisms:
$$\Hom_\O(M,M'):=\{f\in \Hom_\R(M,M')\mid f(px)=pf(x),f(xp)=f(x)p, \text{ for all } x\in M, p\in\O \}.$$
That is, $$\Hom_\O(M,M')=l\text{-}\Hom_\spo(M,M')\cap r\text{-}\Hom_{\spo}(M,M').$$
It turns out that the three sets above are all the same in bimodule case.

\begin{prop}\label{lem:l-Hom=Hom}
	Suppose $M$ and $M'$ are two  $\spo$-bimodules. Then
	\begin{align}
	%\huaa{l\text{-}Hom(M,M')}&=l\text{-}Hom_{\spo}(M,M')\\
	{l\text{-}\Hom_\spo(M,M')}&=r\text{-}\Hom_{\spo}(M,M')=\Hom_{\spo}(M,M')\label{seteq:huaal(lHom(M,M)=HomO(MM))}
	\end{align}
\end{prop}
\begin{proof}
	Let $f\in {l\text{-}\Hom_\spo(M,M')}$. For any $x\in \huaa{M}$, in view of Theorem \ref{thm:general bimod strc}, we deduce
	$$	f(xp)=f(px)	=pf(x),$$
	Since $x\in \huaa{M}$, we conclude that $f(x)\in \huaa{M'}$  and hence $pf(x)=f(x)p$, namely,
	%Lemma \ref{lem:4[ei,x]=[ei,ej,x]} ensures that $x\in \huaal{M}$,  hence we get,
	\begin{align*}
	%\intertext{Therefore, we have Proposition \ref{prop:f(huaaM) in huaaN} yields $f(x)\in \huaa{M'}$, }$	pf(x)=f(x)p.$ Therefore, we have
	f(xp)&=f(x)p, \text{ for all } x\in \huaa{M}.\tag{$\star$}
	\intertext{Now for arbitrary $x\in M$, suppose $\{x_j\}_{j\in \Lambda}\subseteq \huaa{M}$ is a basis of $M$, then we can write $x=\sum_{i=1}^n r_ix_i$, where  $r_i\in \O$, $i=1,\dots,n$. It follows that}
	f(xp)&=\sum_{i=1}^n  f(x_i(r_ip))&\text{since $x_i\in \huaa{M} $}\\
	&\xlongequal{(\star)}\sum_{i=1}^n  f(x_i)(r_ip)\\
	&=\sum_{i=1}^n  (f(x_i)r_i)p&\text{since $f(x_i)\in \huaa{M} $}\\
	&\xlongequal{(\star)}f(x)p.
	\end{align*} 
	This means that $f\in r\text{-}\Hom_{\spo}(M,M') $ and thus $f\in \Hom_{\spo}(M,M') $ as desired.
	Therefore, we obtain that $l\text{-}\Hom_{\spo}(M,M')=\Hom_{\spo}(M,M')$, similarly, it also holds $r\text{-}\Hom_{\spo}(M,M')=\Hom_{\spo}(M,M')$. This completes the proof.
\end{proof}
\begin{rem}
Proposition \ref{lem:l-Hom=Hom}  shows there is no difference between left $\spo$-homomorphisms and  $\spo$-bihomomorphisms when $M$ is an bimodule. Therefore it is no  need to consider them separately as in \cite{ludkovsky2007algebras}. What's more, this proposition  actually  shows that  if $M$ and $M'$ are two $\spo$-bimodules such that they are isomorphic as left $\spo$-modules, then they are isomorphic as  $\spo$-bimodules. 	
\end{rem}
%Proposition \ref{lem:l-Hom=Hom}  shows there is no difference between left $\spo$-homomorphisms and  $\spo$-bihomomorphisms when $M$ is an bimodule. Therefore it is no  need to consider them separately as in \cite{ludkovsky2007algebras}. What's more, this proposition  actually  shows that  if $M$ and $M'$ are two $\spo$-bimodules such that they are isomorphic as left $\spo$-modules, then they are isomorphic as  $\spo$-bimodules. 
%In fact, by the lemma below, we further have
%\begin{thm}\label{thm:free bimod= free mod}
%	The category of free $\spo$-bimodules is isomorphic to  the category of free left $\spo$-modules.
%\end{thm}
%\begin{thm}\label{thm:free bimod= free mod}
%The category of free $\spo$-bimodules is isomorphic to  the category of free left $\spo$-modules.
%\end{thm}

\section{The real part of $\spo$-bimodules}

In this section, we shall introduce the structure of  real part on  $\O$-bimodules in a similar way as in quaternion case \cite{ng2007quaternionic}. It turns out the category of $\O$-bimodules is also isomorphic to the category of $\R$-vector spaces.
%\subsection{title} $\re M$ 

Let $M$ be an $\O$-bimodule. For any given $x\in M$, since $M=\O \huaa{M}$, %\cong \O\otimes\huaa{M}
then for any $m\in M$, we can write $$m=\sum_{i=1}^nr_ix_i, $$ for some $ r_i\in \O$ and some $x_i\in \huaa{M}$.
Let $r_i=r_{i0}+\sum r_{ij}e_j$, $r_{ij}\in \R$ for $j=0,\ldots,7$ and  $i=1,\ldots,n$. Hence the above equality can be  rewritten as  $$m=m_0+\sum_{i=1}^7 e_im_i,$$ for some associative elements $m_j\in \huaa{M},\; j=0,1,\dots, 7.$

% there exsits a unique decomposition: 
%$$m=m_0+\sum e_im_i,\quad m_j\in \huaa{M},\ j=0,1,\dots, 7.$$  We call $m_0$ the \textbf{real part} of $m$, denoted by $\re m$. It's easy seen that   $m_i=-\re(e_im).$
%hence there exist $x_i\in \huaa{M}$ for $i=0,1,\ldots,7$ such that 
%$$x=x_0+\sum e_ix_i.$$
%We call $x_0$ the \textbf{real part} of $x$, denoted by $\re x$. 
%We shall give $\re x$ a concrete formmula. The following lemma is crucial.
\begin{lemma}\label{lem:4[ei,x]=[ei,ej,x]}
	Let $M$ be an $\spo$-bimodule. Then we have the following equations:
	\begin{eqnarray}\label{eq:4[ei,x]=sum eijk[ej,ek,x]}
	4[e_i,x]=\sum_{j,k} \epsilon_{ijk}[e_j,e_k,x],\quad i=1,\ldots,7.
	\end{eqnarray}
\end{lemma}
\begin{proof}
	%Since $M$ is \kelie\ $\spo$-bimodule, so for arbitrary $x\in M$, we can write
	For any given $x\in M$, let
	$$x=x_0+\sum e_ix_i,\quad \text{ where } x_j\in \huaa{M},\ j=0,1,\dots, 7.$$ 
	We first compute $\sum_{j,k} \epsilon_{ijk}[e_j,e_k,x]$. Using Einstein summation convention,
	\begin{align*}
	\sum_{j,k} \epsilon_{ijk}[e_j,e_k,x]&= \epsilon_{ijk}[e_j,e_k,e_mx_m]\\
	&=%\xlongequal{lemma \eqref{lem: [p,q,rm]=[p,q,r]m}}
	\epsilon_{ijk}[e_j,e_k,e_m]x_m\\
	&\xlongequal{ \eqref{eq:e(ijkl)el=[ei,ej,ek]}}2\epsilon_{ijk}\epsilon_{ikmn}e_nx_m\\
	&\xlongequal{ \eqref{eq:contract id eijq eijkl=4eqkl}}8\epsilon_{imn}e_nx_m.
	\end{align*}
	Whereas
	\begin{align*}
	4[e_i,x]&=4[e_i,e_mx_m]\\
	&=4\left((e_ie_m)x_m-e_m(x_me_i)\right)&\text{$x_m\in \huaa{M}$}\\
	&=4\left(e_ie_m-e_me_i\right)x_m&\text{$x_m\in \hua{Z}{}{M}$}\\
	&=4(\epsilon_{imn}e_n-\epsilon_{min}e_n)x_m\\
	&=8\epsilon_{imn}e_nx_m.
	\end{align*}
	This proves the equations \eqref{eq:4[ei,x]=sum eijk[ej,ek,x]} as desired.
	%	 and hence we conclude that 
	%	$$ \huaa{M}\subseteq \huaal{M}.$$
	%	Therefore we obtain $\huaal{M}=\huaa{M}$ as desired.	
\end{proof}
\begin{cor}
	Let $M$ be an $\O$-bimodule. Then the right multiplication is uniquely determined by its left module structure. More precisely, for any $x\in M$, the right multiplication is given by
	$$xe_i=e_ix-\dfrac{1}{4}\sum_{j,k} \epsilon_{ijk}[e_j,e_k,x],\quad i=1,\ldots,7.$$
\end{cor}
\begin{rem}
	This lemma gives a new proof of  the fact that $\huaa{M}\subseteq \hua{Z}{}{M}$.
\end{rem}

We wish to define a real part structure on $\O$-bimodule $M$, which plays a role of the real number in $\O$.  Recall in the quaternion setting, Ng gives a  structure of real part on a quaternion bimodule $X$  as follows (\cite{ng2007quaternionic}):
$$\re x=\frac{1}{4}\sum _{e\in B}\overline{e}xe \quad (x\in X).$$
Where $B := \{1,i, j, k\}$ is a basis of the quaternions $\mathbb{H}$.

In the octonion case, it turns out that there is a   real part structure on $\O$-bimodules as well. We define the \textbf{real part} for an arbitrary  bimodule $M$ as  follows: $$\vre x:=\frac{5}{12}x-\frac{1}{12}\sum_{i=1}^7 e_ixe_i\quad (x\in M).$$ 
The operator $\re: M\rightarrow M$ is called the \textbf{real part structure} of $M$, and we call an element $m\in M$ \textbf{real}  if $m\in \re M$.
Combining the equations \eqref{eq:4[ei,x]=sum eijk[ej,ek,x]},  we conclude that $\vre x=x+\dfrac{1}{48}\epsilon_{ijk}e_i[e_j,e_k,x]$. It turns out that the subset of all real elements coincides with the   subset $\huaa{M}$ of associative elements. 
%serves as the real part of $M$.
\begin{thm}\label{lem:real part on good bimod }
	If $M$ is an $\O$-bimodule, then for all $x\in M$, the following hold:
	\begin{enumerate}
		\item 	$\vre^2x=\vre x$.
		\item  $x=\vre x-\sum e_i\vre(e_ix)$.% and $M=\vre (M)\bigoplus \oplus_{i=1}^7e_i\vre(M)$ .
		\item $\vre M= \hua{Z}{}{M}=\huaa{M}$.
		\item For all $x\in \re M$, $p\in \O$, we have $\re (px)=(\re p)x$.
		\item $M=\vre M\oplus \bigoplus_{i=1}^7e_i\vre M$. %$\re x=\vre x$.
	\end{enumerate}
	
\end{thm}
\begin{proof}
	We prove a\asertion{1}. The proof is straightforward. 
	\begin{align*}
	\vre^2x&=\vre(\frac{5}{12}x-\frac{1}{12}\sum e_ixe_i)\\
	&=\frac{5}{12}(\frac{5}{12}x-\frac{1}{12} e_ixe_i)-\frac{1}{12}e_j(\frac{5}{12}x-\frac{1}{12} e_ixe_i)e_j\\
	&=\frac{5^2}{12^2}x-\frac{10}{12^2}e_ixe_i+\frac{1}{12^2}e_j(e_ixe_i)e_j.
	\end{align*}
	Using Moufang identities, we obtain 
	\begin{align*}
	e_j(e_ixe_i)e_j&=\Big(\big((e_je_i)x\big)e_i\Big)e_j\\
	&=\Big(\big((\epsilon_{jik}e_k-\delta_{ji})x\big)e_i\Big)e_j\\
	&=\big((\epsilon_{jik}e_k-\delta_{ji})x\big)(e_ie_j)+[(\epsilon_{jik}e_k-\delta_{ji})x,e_i,e_j]\\
	&=(\epsilon_{jik}e_kx-\delta_{ji}x)(\epsilon_{ijm}e_m-\delta_{ji})+\epsilon_{jik}[e_kx,e_i,e_j]\\
	&=-6\delta_{km}(e_kx)e_m+7x+\epsilon_{jik}[e_i,e_j,e_kx]
	\end{align*}
	where  we have used identity \eqref{eq:ep(ijk)ep(ijl)=6delta(kl)} in the last line. Note that 
	\begin{align*}
	\epsilon_{jik}[e_kx,e_i,e_j]&=	\epsilon_{jik}[e_i,e_j,e_kx]\\
	&=\epsilon_{jik}([e_i,e_j,e_k]x+e_i[e_j,e_k,x]-[e_ie_j,e_k,x]+[e_i,e_je_k,x])&\text{by identity \eqref{eq:[p,q,r]m+p[q,r,m]=[pq,r,m]-[p,qr,m]+[p,q,rm]}}\\
	&=\epsilon_{jik}(2\epsilon_{ijkm}e_mx+e_i[e_j,e_k,x]-\epsilon_{ijm}[e_m,e_k,x]+\epsilon_{jkm}[e_i,e_m,x])\\
	&=\epsilon_{jik}e_i[e_j,e_k,x]&\text{note that $\epsilon_{jik}\epsilon_{ijkm}=0$}\\
	&=-4e_i[e_i,x]\\
	&=\sum_{i=1}^7-4e_i(e_ix-xe_i)\\
	&=28x+4e_ixe_i,
	\end{align*}
	where we have used the equations \eqref{eq:4[ei,x]=sum eijk[ej,ek,x]}.
	Hence 
	\begin{align*}
	\vre^2x&=\frac{5^2}{12^2}x-\frac{10}{12^2}e_ixe_i+\frac{1}{12^2}(-6\delta_{km}(e_kx)e_m+7x+28x+4e_ixe_i)\\
	&=\frac{5}{12}x-\frac{1}{12}\sum e_ixe_i\\
	&=\vre x.
	\end{align*}
	This proves a\asertion{1}.
	
	We prove a\asertion{2}. 
	First using Moufang identities again, we get 
	\begin{align*}
	e_j(e_ix)e_j=(e_je_i)(xe_j)=\epsilon_{jik}e_k(xe_j)-xe_i.
	\end{align*}
	Hence
	\begin{align*}
	\vre(e_ix)&=\frac{5}{12}e_ix-\frac{1}{12}e_j(e_ix)e_j\\
	&=\frac{5}{12}e_ix-\frac{1}{12}(\epsilon_{jik}e_k(xe_j)-xe_i)
	\end{align*}
	and 
	\begin{align*}
	\vre x-\sum e_i\vre(e_ix)&=\frac{5}{12}x-\frac{1}{12}\sum e_ixe_i-e_i\big(\frac{5}{12}e_ix-\frac{1}{12}(\epsilon_{jik}e_k(xe_j)-xe_i)\big)\\
	&=\frac{40}{12}x-\frac{2}{12}e_ixe_i+\frac{1}{12}\epsilon_{jik}e_i\big(e_k(xe_j)\big)
	\end{align*}
	Similar as above,
	\begin{align*}
	\epsilon_{jik}e_i\big(e_k(xe_j)\big)&=\epsilon_{jik}\big((e_ie_k)(xe_j)-[e_i,e_k,xe_j]\big)\\
	&=\epsilon_{jik}\big((\epsilon_{ikm}e_m-\delta_{ik})(xe_j)-[xe_j,e_i,e_k]\big)\\
	&=6\delta_{jm}e_m(xe_j)-\epsilon_{jik}[x,e_j,e_i]e_k\\
	&=6e_ixe_i-4[e_k,x]e_k\\
	&=2e_ixe_i-28x.
	\end{align*}
	Hence 
	\begin{align*}
	\vre x-\sum e_i\vre(e_ix)=x.
	\end{align*}
	%Let $x=\sum_{i=0}^7e_ix_i$, $x_i\in \vre(M)$ for $i=0,\ldots,7$.
	
	We prove a\asertion{3}. 
	\begin{align*}
	\vre x\in \hua{Z}{}{M}&\iff e_j \vre x=(\vre x) e_j, \ j=1,\ldots,7\\
	&\iff \frac{5}{12}[e_j,x]=\frac{1}{12}\left((e_ixe_i)e_j-e_j(e_ixe_i)\right), \ j=1,\ldots,7\\
	&\iff \frac{5}{12}[e_j,x]=\frac{1}{12}\left(e_i(x(e_ie_j))-((e_je_i)x)e_i\right), \ j=1,\ldots,7\\
	&\iff 4[e_j,x]=\sum_{i,k} \epsilon_{jik}[e_i,e_k,x],\quad j=1,\ldots,7.
	\end{align*}
	Then it follows from  Lemma \ref{lem:4[ei,x]=[ei,ej,x]} that  $\vre M\subseteq\hua{Z}{}{M}$. On the other hand,   for any $x\in \hua{Z}{}{M}$, $\vre x=\frac{5}{12}x-\frac{1}{12}\sum e_ixe_i=x$ and thus $\hua{Z}{}{M}\subseteq \vre M$. Hence $\vre M= \hua{Z}{}{M}=\huaa{M}$.
	This proves a\asertion{3}.
	
	We prove a\asertion{4}. Note that $\vre M=\huaa{M}$ and $\vre x=x+\dfrac{1}{48}\epsilon_{ijk}e_i[e_j,e_k,x]$, we conclude for every $x\in \vre M$, $m=1,\dots,7$,
	\begin{align*}
	\vre(e_mx)&=e_mx+\dfrac{1}{48}\epsilon_{ijk}e_i[e_j,e_k,e_m]x\\
	&=e_mx+\dfrac{1}{48}\epsilon_{ijk}\cdot 2\epsilon_{jkmn}(\epsilon_{inl}e_l-\delta_{in})x\\
	&=e_mx+\dfrac{1}{6}\epsilon_{mni}\epsilon_{inl}e_lx\\
	&=e_mx-\dfrac{1}{6}\cdot 6\delta_{ml}e_lx\\
	&=0
	\end{align*}
	where we have used identities \eqref{eq:ep(ijk)ep(ijl)=6delta(kl)} and \eqref{eq:contract id eijq eijkl=4eqkl}. This proves a\asertion{4}.
	%\epsilon_{ijq}\epsilon_{ijkl}&=4\epsilon_{qkl}\label{eq:contract id eijq eijkl=4eqkl}
	
		 A\asertion{5} follows from a\asertion{2} and a\asertion{4 } directly. 
%	Hence if $x=x_0+\sum_{i=1}^7e_ix_i$ for some $x_i\in \vre M$, $i=0,\ldots,7$, then 
%	$$\vre(e_ix)=\vre(e_ix_0)+\sum_{j=1}^7 \vre(e_i(e_jx_i))=-x_i.$$ 
%	This shows that the decomposition is unique. In view of a\asertion{2}, we get the conclusion as desired.
%	%	By a\asertion{2} and a\asertion{3}, we claim  $M=\vre (M)\bigoplus \oplus_{i=1}^7e_i\vre(M)$
%	%	
%	%	 we know that we have a decomposition $x= \vre x-\sum e_i\vre(e_ix) $ with $\vre(e_ix)\in \vre{M}=\huaa{M}$ for $i=0,1,\ldots,7$. Hence it follows from the uniqueness of the decomposition that $\re x=\vre x$. 
\end{proof}
An immediate consequence is a concrete form of the subset of associative elements, which only depends on the left multiplication.
\begin{cor}
	Let $M$ be an $\O$-bimodule, then we have $$\huaa{M}=\left\lbrace x+\dfrac{1}{48}\epsilon_{ijk}e_i[e_j,e_k,x] \middle|\; x\in M  \right\rbrace  .$$
\end{cor}

For any given $m\in M$, it follows from the above theorem that 
there exsits a unique decomposition: 
$$m=m_0+\sum e_im_i,$$
where $ m_j=\vre(\overline{e_i}x)\in \huaa{M}$, for all $ j=0,1,\dots, 7$.   We call $m_0=\vre m$ the \textbf{real part} of $m$. 
%, denoted by $\re m$. That is, $\re m=\vre m$. 
 $\re M$ is the unique real vector space (up to isomorphism) whose octonionization, $\re M\otimes \O$, is isomorphic to $M$ as $\O$-bimodule. 
  %By the above theorem,    $m_i=-\re(e_im).$
  We can show that the properties of this real part structure are almost the same with that in octonion. We assemble some elementary properties now.
  \begin{prop}
  	Let $M$ be an $\O$-bimodule. Then for all $p,q\in \O$, for all $x\in M$, we have 
  	\begin{enumerate}
  		
  		\item $\re [p,q,x]=0$;
  		\item $\re [p,x]=0$;
  		\item $\re (pq)x=\re (qx)p=\re x(pq)$.
  	\end{enumerate}
  \end{prop}
\begin{proof}
	Suppose $x=x_0+\sum e_ix_i$, $x_j\in \re M$ for $j=0,1,\dots,7.$
	Then in view of identity \eqref{eq:[p,q,r]m+p[q,r,m]=[pq,r,m]-[p,qr,m]+[p,q,rm]} and a\asertion{4 } in Theorem \ref{lem:real part on good bimod }, we obtain
	\begin{align*}
	\re [p,q,x]=\re [p,q,\sum_0^7 e_ix_i]=\sum_0^7\re ([p,q, e_i]x_i)=\sum_0^7(\re [p,q, e_i]) x_i=0.
	\end{align*}
	This proves a\asertion{1}. A\asertion{2} follows from the a\asertion{1} that we just proved and Lemma \ref{lem:4[ei,x]=[ei,ej,x]}. We prove a\asertion{3}.
	Thanks to a\asertion{2}, we get $$\re(pq)x=\re(pq)x-\re [pq,x]=\re x(pq),$$
	and $$\re (pq)x=\re(pq)x-\re [p,q,x]= \re p(qx)=\re(qx)p.$$
	This proves the proposition.
\end{proof}
  Similar as  in quaternion case (\cite{ng2007quaternionic}), we can prove the following lemma. The proof is much the same and  is omitted here.
\begin{lemma}
	Let $X$ and $Y$ be two $\O$-bimodules. If $f\in \Hom_{\O} (X,Y)$, then $f|_{\re(X)}\in \Hom_{\R}\big(\re X,\re Y\big)$. This induces an $\R$-linear isomorphism $\Psi$ from $\Hom_{\O} (X,Y)$ onto $\Hom_{\R}\big(\re X,\re Y\big)$.
	
\end{lemma}
In exactly the same way as  quaternion case, it holds:
\begin{thm}
	The category of $\O$-bimodules is isomorphic to the category of $\R$-vector spaces.
\end{thm}

\begin{rem}
	The notion of real part on an $\O$-bimodule will play a crucial role in the development of octonionic functional analysis. We hope to indicate some of these applications in subsequent papers.
\end{rem}
\section{Submodules generated by one element}
The submodule $\left\langle m\right\rangle _{\O}$ generated  by one point  will be very different  from the classical case.  As is known that $\O m$ is not always a submodule  and the submodule generated by one element may be the whole module \cite{goldstine1964hilbert}. 
 We introduce the notion of cyclic element and \decomp\    to describe this phenomenon.

\subsection{Cyclic elements}
Let $M$ be a left $\spo$-module only in this subsection. 
We collect some basic properties  and lemmas first.

%We first review the definition and some properties of cyclic elemets in \cite{liyong2019octonionmodule}.
\begin{mydef}\label{def:cir,ass}
	An element 	$m\in M$ is said to be \textbf{cyclic} if $\left\langle m\right\rangle _\O=\O m$.
	Denote by  $\huac{M}$ the set of all cyclic elements in $M$.
\end{mydef}

The following lemma is crucial to set up the structure of $\huac{M}$.
\begin{lemma}\label{lem:x in huac(M)=dim =8}
	Let $x$ be any given nonzero element in $M$. Then $$x\in \huac{M}\iff \mathrm{dim}_\spr \generat{x}_\spo=8\iff \generat{x}_\spo\cong \spo \text{ or } \overline{\spo}.$$
\end{lemma}
\begin{proof}
	Let $x\in \huac{M}$, then $\generat{x}_\spo=\spo x$ and hence $\mathrm{dim}_\spr \generat{x}_\spo\leqslant8$. On the other hand, since $\generat{x}_\spo$ is a nonzero  $\spo $-module of finite dimension, thus $\mathrm{dim}_\spr \generat{x}_\spo\geqslant8$, therefore $\mathrm{dim}_\spr \generat{x}_\spo=8$, this means  $ \generat{x}_\spo$ is a simple $\spo$-module and hence $\generat{x}_\spo\cong \spo \text{ or } \overline{\spo}$. Suppose $\generat{x}_\spo\cong \spo \text{ or } \overline{\spo}$. Assume $\generat{x}_\spo\cong \spo $ first. Let $\varphi$ denote an isomorphism: $\varphi:\generat{x}_\spo\rightarrow \spo.$ For any $m\in \generat{x}_\spo$, suppose $\varphi(x)=p,\ \varphi(m)=q$. Then
	$$\varphi(m)=q=qp^{-1}p=qp^{-1}\varphi(x)=\varphi((qp^{-1})x),$$
	according to that $\varphi $ is isomorphism, we  get $m=(qp^{-1})x$, hence $\generat{x}_\spo=\spo x$ which means $x\in \huac{M}$.  If $\generat{x}_\spo\cong \overline{\spo }$,  still let $\varphi$ denote the isomorphism: $\generat{x}_\spo\rightarrow \overline{\spo }.$ For any $m\in \generat{x}_\spo$, suppose $\varphi(x)=p,\ \varphi(m)=q$. Then
	$$\varphi(m)=q=\overline{(qp^{-1})}\hat{\cdot}p=\overline{(qp^{-1})}\hat{\cdot}\varphi(x)=\varphi(\overline{(qp^{-1})}x),$$
	then we  get $m=\overline{(qp^{-1})}x$, hence  $x\in \huac{M}$.
\end{proof}

%Now let us first consider the cyclic elements in finite dimensional $\O$-modules. It had been seen that $(e_1,e_2,0)$ will generat a $16$ dimensional real space and $(e_1,e_2,e_3)$ will generat a $24$ dimensional real space (\cite{goldstine1964hilbert}), which means that  both are   not cyclic elements.  In fact, this is a general  phenomenon.

According to the above lemma, we define $$\huacc{+}{M}:=\{x\in \huac{M}\mid \generat{x}_\spo\cong \spo \}\cup\{0\},$$  $$\huacc{-}{M}:=\{x\in \huac{M}\mid \generat{x}_\spo\cong \overline{\spo}\}\cup\{0\}.$$ Therefore
$\huac{M}=\huacc{+}{M}\cup\huacc{-}{M}.$ 
We shall show that all the cyclic elements are determined by the associative subset $\huaa{M}$ and the conjugate associative subset $\hua{A}{-}{M}$.
%	\begin{cor}\label{cor: px in huac(M)}
%		$\huac{M}$ is closed under  multiplication. In other words, $\huac{M}=\bigcup _{x\in \huac{M}}\generat{x}_\spo$. So are $\huacc{+}{M}$ and $\huacc{-}{M}$.
%	\end{cor}
%	\begin{proof}
%		Obviously $x\in {\generat{px}_\spo}$ and $px\in \generat{x}_\spo$, this implies $\generat{px}_\spo=\generat{x}_\spo$, hence by  lemma \ref{lem:x in huac(M)=dim =8}, we obtain $px\in \huac{M}$. The proof of the other cases runs as the same way.
%	\end{proof}
%\begin{rem}\label{rem:C+ def of phi x}
%	Let $0\neq x\in \huacc{+}{M}$,   in consideration of example \ref{eg:Hom}, we obtain
%	$$Hom_\spo(\spo,\generat{x}_\spo)\cong Hom_\spo(\spo,\spo)\cong \spr.$$
%%	Hence there exists a unique isomorphism  $\varphi_x:\generat{x}_\spo\rightarrow \spo$ such that $$|\varphi_x(x)|=1,\quad \re(\varphi_x(x))\geqslant0$$
%	Along with proposition \ref{prop:f(huaaM) in huaaN}, the  homomorphisms in $ Hom_\spo(\spo,\generat{x}_\spo)$ induce a corresponding: $$\tau:\huacc{+}{M}\rightarrow \text{the set of real subspaces of }\huaa{M}, \quad x\mapsto \tau(x):=\{\varphi(1)\mid \varphi\in Hom_\spo(\spo,\generat{x}_\spo)\}.$$
%	Clearly, this is well defined and $x\in \tau(x)$ if and only if $x\in \huaa{M}$. In fact, we can say more about this corresponding.
%\end{rem}

\begin{thm}\label{lem:huac+(M)}
	Let $M$ be a left $\spo$-module, then:
	\begin{enumerate}%[label=({\roman*}).]
		\item $\huacc{+}{M}=\bigcup_{p\in \spo}p\cdot \huaa{M}$;
		\item  $\huacc{-}{M}=\bigcup_{p\in \spo}p\cdot \hua{A}{-}{M}$.
		%	Span_\spr\{\tau(x)\mid x\in \huacc{+}{M}\}$.		
	\end{enumerate}
\end{thm}

\begin{proof}
	We prove a\asertion{1}.  We first show $\bigcup_{p\in \spo}p\cdot \huaa{M}\subseteq\huacc{+}{M}$.  Given any $x\in \huaa{M}$. \bfs\ $x\neq 0$.
	Define a map $\phi:\generat{x}_\spo\rightarrow \spo$ such that $ \phi( px)= p$ for $p\in \spo$.
	This is a homomorphism in $\Hom_\spo(\generat{x}_\spo, \spo)$, since $$\phi(q(px))=\phi((qp)x)=qp=q\phi(px).$$
%	Define $\varphi:\spo\rightarrow \generat{x}_\spo$ by $ \varphi(p)= px$. Then
%	$$\varphi(pq)=(pq)x=p(qx)=p\varphi(q).$$ Hence $\varphi\in \Hom_\spo( \spo,\generat{x}_\spo)$ and $\phi \varphi=id,\varphi\phi=id$ and thus
 Thus $\generat{x}_\spo\cong \spo$. This proves $x\in \huacc{+}{M}$. Because $px\in \generat{x}_\spo$ and $x=p^{-1}(px)\in \generat{px}_\spo$ for $p\neq 0$, that is, $\generat{x}_\spo=\generat{px}_\spo$ whenever $p\neq 0$. This implies $\bigcup_{p\in \spo}p\cdot \huaa{M}\subseteq\huacc{+}{M}$.
	On the contary, let $0\neq x\in \huacc{+}{M}$, hence there is an isomorphism $\phi\in \Hom_\spo(\spo,\generat{x}_\spo)$. Suppose $\phi (1)=y\in \generat{x}_\spo$, since $\phi$ is an isomorphism, choose $0\neq r\in \spo$ such that $y=rx$. Note that $[p,q,y]=\phi[p,q,1]=0$ for all $p,q \in \O$, we thus get
	% in view of Proposition \ref{prop:f(huaaM) in huaaN},
	  $y=\phi (1)\in \huaa{\generat{x}_\spo}\subseteq \huaa{M}$, and hence $x=r^{-1}y\in \bigcup_{p\in \spo}p\cdot \huaa{M}$. This proves a\asertion{1}.
	Similarly we can prove a\asertion{2}.	
%	We prove a\asertion{2}. Easy to show that $\generat{x}_\spo=\spo x$ for $x\in \hua{A}{-}{M}$.
%	Let $x\in \hua{A}{-}{M}$. \bfs\ $x\neq 0$.
%	Define a map $\phi:\generat{x}_\spo\rightarrow \overline{\spo}$ such that $ \phi( px)=\overline{ p}$ for $p\in \spo$.
%	This is a homomorphism in $\Hom_\spo(\generat{x}_\spo, \spo)$, since $$\phi(q(px))=\phi((pq)x)=\overline{ pq}=q\hat{\cdot}\overline{ p}=q\phi(px).$$
%	Define $\varphi:\overline{\spo}\rightarrow \generat{x}_\spo$ by $ \varphi(p)= \overline{p}x$. Then
%	$$\varphi(p\hat{\cdot}q)=\varphi(\overline{p}q)=\overline{(\overline{ p}q)}x={ p}(\overline{q}x)=p\varphi(q).$$ Hence $\varphi\in \Hom_\spo( \spo,\generat{x}_\spo)$ and $\phi \varphi=id,\varphi\phi=id$ and thus $\generat{x}_\spo\cong \overline{\spo}$. This proves $x\in \huacc{-}{M}$. Hence we conclude from the fact $\generat{x}_\spo=\generat{px}_\spo$ that  $\bigcup_{p\in \spo}p\cdot \hua{A}{-}{M}\subseteq\huacc{-}{M}$.
%	On the contary, let $0\neq x\in \huacc{-}{M}$, hence there is an isomorphism $\phi\in \Hom_\spo(\overline{\spo},\generat{x}_\spo)$. Suppose $\phi (1)=y\in \generat{x}_\spo$, since $\phi$ is an isomorphism, there is $0\neq r\in \spo$ such that $y=rx$. Note that for any $p,q\in \spo$,
%	$$(pq)y=(pq)\phi(1)=\phi((pq)\hat{\cdot}1)=\phi(q\hat{\cdot} \overline{p})=q\phi(p\hat{\cdot}1)=q(py).$$
%	This shows $y\in \hua{A}{-}{M}$ and
%	thus $x=r^{-1}y\in \bigcup_{p\in \spo}p\cdot \hua{A}{-}{M}$. This proves a\asertion{2}.	
\end{proof}

The following lemma will be useful later. The proof can be found in \cite{liyong2019octonionmodule}.
\begin{lemma}\label{lem: xishu asselm}
	Let $ \{x_i\}_{i=1}^n$ be an  $\spr$-linearly independent set of \asselm s of  $M$, then $ \{x_i\}_{i=1}^n$ is also  $\spo$-linearly independent. Further if $y=\sum_{i=1}^nr_ix_i\in \huaa{M}$, then $r_i\in \spr  \text{ for each }i\in \{1,\dots, n\}.$
\end{lemma}

\subsection{The \decomp}
%Let $M$ be a left $\spo$-module. For any nonzero element $m\in \huac{M}$, it follows from Theorem \ref{lem:huac+(M)} that there exists $p\in \O$ such that $m=px_m$ for some real vector $x_m\in \huaa{M}\cup \hua{A}{-}{M}$. If we have another $p'\in\O$ and $x'\in \huaa{M}\cup \hua{A}{-}{M}$, such that $m=p'x'$. \bfs\ $m\in \hua{C}{+}{M}$ and hence $x_m,x'\in \huaa{M}$, since $m\neq 0$, we obtain $$x_m=p^{-1}(p'x')=(p^{-1}p')x'.$$
%Thanks to Lemma \ref{lem: xishu asselm}, we conclude $p^{-1}p'\in \R$. That is, viewed $\O$ as the real vector space $\R^8$, the vectors $p$ and $p'$ are parallel.  This induces a map $$\sigma:\huac{M}\rightarrow \mathbb{PR}^8 \qquad m\mapsto [p].$$ Such octonion $p$ is called a \textbf{\tezzh} of $m$ and the real vector $x_m$ is called a \textbf{\tezh} of $m$. Denote by $\mathbb{P}_m$ the collections of \tezh s of $m$. 

In this subsction, we will introduce a notion of \textbf{\decomp}. With the help of this concept, we will formulate the structure of the submodule generated by one element.  Let $M$ be an $\O$-bimodule throughout this subsection. It follows from Theorem \ref{thm:general bimod strc} and Theorem \ref{lem:huac+(M)} that $M=\text{Span}_\R\huac{M}$.
%   and   $\huac{M}=\hua{C}{+}{M}$. 

For any nonzero element $m\in \huac{M}$,  Theorem assures that \ref{lem:huac+(M)} that there exists an octonion $p\in \O$ such that $m=px_m$ for some associative element $x_m\in \huaa{M}$. 
Such an octonion $p$ is called a \textbf{\tezzh} of $m$ and the real vector $x_m$ is called a \textbf{\tezh} of $m$. 
%Denote by $\mathbb{P}_m$ the collections of \tezh s of $m$.

If we have another $p'\in\O$ and $x'\in \huaa{M}$, such that $m=p'x'$. Note that $m\neq 0$, we thus  obtain $$x_m=p^{-1}(p'x')=(p^{-1}p')x'.$$
Thanks to Lemma \ref{lem: xishu asselm}, we conclude $p^{-1}p'\in \R$. That is, viewed $\O$ as the real vector space $\R^8$, the vectors $p$ and $p'$ are parallel.  
This induces a map $$\sigma\colon\huac{M}\to \mathbb{R}P^8 \qquad m\mapsto [p].$$

%Such an octonion $p$ is called a \textbf{\tezzh} of $m$ and the real vector $x_m$ is called a \textbf{\tezh} of $m$. Denote by $\mathbb{P}_m$ the collections of \tezh s of $m$. 

%For any nonzero element $m\in \huac{M}$, it follows from Theorem \ref{lem:huac+(M)} that there exists $p\in \O$ such that $m=px_m$ for some real vector $x_m\in \huaa{M}\cup \hua{A}{-}{M}$. If we have another $p'\in\O$ and $x'\in \huaa{M}\cup \hua{A}{-}{M}$, such that $m=p'x'$. \bfs\ $m\in \hua{C}{+}{M}$ and hence $x_m,x'\in \huaa{M}$, since $m\neq 0$, we obtain $$x_m=p^{-1}(p'x')=(p^{-1}p')x'.$$
%Thanks to Lemma \ref{lem: xishu asselm}, we conclude $p^{-1}p'\in \R$. That is, viewed $\O$ as the real vector space $\R^8$, the vectors $p$ and $p'$ are parallel.  This induces a map $$\sigma:\huac{M}\rightarrow \mathbb{PR}^8 \qquad m\mapsto [p].$$ Such octonion $p$ is called a \textbf{\tezzh} of $m$ and the real vector $x_m$ is called a \textbf{\tezh} of $m$. Denote by $\mathbb{P}_m$ the collections of \tezh s of $m$. 

\begin{mydef}\label{def:decomp}
	Let $m\in M$ be any given element. Let $m=\sum _{i=1}^n m_i=\sum _{i=1}^n r_ix_i$ be a decomposition of cyclic elements, where $\{x_i\}_{i=1}^n$ and $\{r_i\}_{i=1}^n$ are the 
	collection of corresponding  \tezh s and  corresponding  \tezzh s respectively.
	Then it  is called a \textbf{\decomp} of $m$ if it satisfies:
	\begin{enumerate}
		%\item $m=\sum _{i=1}^n m_i $;
		\item %the collection of corresponding  \tezh s
		 $\{x_i\}_{i=1}^n$ is $\R$-linearly independent;
		\item %the collection of corresponding  \tezzh s
		 $\{r_i\}_{i=1}^n$ is $\R$-linearly independent in $\O$.
	\end{enumerate}
\end{mydef}

The following lemma guarantees the existence of \decomp.
\begin{lemma}
	%Each element in $\text{Span}_\R\hua{C}{+}{M}$ has a \decomp.
		Each element in an $\O$-bimodule has a \decomp.
\end{lemma}
\begin{proof}
	The proof will be divided into three steps. Let $m$ be any given element in $M$.% of $\O$-module $M$.
	\begin{step }
		There is a collection $\{m_i\}_{i=1}^n\subseteq\huac{M}$, such that $m=\sum_{i=1}^nm_i$.
	\end{step }
	This follows from the fact that $M=\text{{Span}}_\R\huac{M}$.
\begin{step }
	We can assume the collection of corresponding  \tezh s $\{x_i\}_{i=1}^n$ is $\R$-linearly independent.
	
\end{step }
	By Step 1, we can assume $m=\sum_{i=1}^nm_i$, and $m_i=r_ix_i$ for $i=1,\ldots,n$.	Suppose that $\{x_i\}_{i=1}^n$ is $\R$-linearly dependent.
\bfs $$\sum_{i=1}^{n-1}t_ix_i=x_n, \quad t_i\in \R\text{ for }i=1,\ldots,n-1.$$
Set $$r_i'=r_i+t_ir_n,\; x_i'=x_i,\; m_i'=r_i'x_i', \qquad \text{for } i=1,\ldots,n-1.$$
Obviously $m_i'\in \huac{M}$ with corresponding  \tezh s $x_i'$ and \tezzh s $r_i'$. It is easy to verify that $$\sum_{i=1}^{n-1}m_i'=\sum_{i=1}^{n-1}(r_i+t_ir_n)x_i=\sum_{i=1}^{n-1}r_ix_i+\sum_{i=1}^{n-1}t_ir_nx_i=\sum_{i=1}^{n}m_i=m.$$
Note that $\{x_i'\}_{i=1}^{n-1}$ is a subset of $\{x_i\}_{i=1}^n$ and hence we can preceed in this way until the collection of  corresponding  \tezh s turns into  an $\R$-linearly independent set.

\begin{step }
	 We can assume the collection of corresponding  \tezzh s $\{r_i\}_{i=1}^n$ is $\R$-linearly independent.
	
\end{step }
	By Step 2, we can assume $m=\sum_{i=1}^nm_i$,  $m_i=r_ix_i$ for $i=1,\ldots,n$,  and $\{x_i\}_{i=1}^n$ is $\R$-linearly independent.
Suppose that $\{r_i\}_{i=1}^n$ is $\R$-linearly dependent.
\bfs $$\sum_{i=1}^{n-1}t_ir_i=r_n, \quad t_i\in \R\text{ for }i=1,\ldots,n-1.$$
Set $$x_i'=x_i+t_ix_n,\; r_i'=r_i,\; m_i'=r_i'x_i', \qquad \text{for } i=1,\ldots,n-1.$$
Obviously $m_i'\in \huac{M}$ with corresponding  \tezh s $x_i'$ and \tezzh s $r_i'$. It is easy to verify that 
$$\sum_{i=1}^{n-1}m_i'=\sum_{i=1}^{n-1}r_i(x_i+t_ix_n)=\sum_{i=1}^{n-1}r_ix_i+\sum_{i=1}^{n-1}t_ir_ix_n=\sum_{i=1}^{n}m_i=m.$$
We claim  $\{x_i'\}_{i=1}^{n-1}$ is also an $\R$-linearly independent set. Indeed, let $\sum_{i=1}^{n-1}s_ix_i'=0$  for some $s_i\in \R,\; i=1,\ldots,n-1$. Then we obtain
$$0=\sum_{i=1}^{n-1}s_i(x_i+t_ix_n)=\sum_{i=1}^{n-1}s_ix_i+(\sum_{i=1}^{n-1}s_it_i) x_n,$$
it then follows from the $\R$-linear independence of  $\{x_i\}_{i=1}^n$ that $s_i=0$ for each $i=1,\ldots,n-1$. This shows that the collection $\{x_i'\}_{i=1}^{n-1}$ is  $\R$-linear independent.
Note that $\{r_i'\}_{i=1}^{n-1}$ is a subset of $\{r_i\}_{i=1}^n$ and hence we can preceed in this way until the collection of corresponding  \tezzh s turns into  an $\R$-linearly independent set.
\end{proof}

\begin{lemma}\label{lem:m1+m2}
	Let $m_i\in \huac{M}$ and $x_i$ be the \tezh s of $m_i$ for $i=1,2$. Then the following hold:
	\begin{enumerate}
		\item if $\{x_1,x_2\}$ is $\R$-linearly dependent, then $m_1+m_2\in \huac{M}$;
		\item if $\{x_1,x_2\}$ is $\R$-linearly independent, then $m_1+m_2\in \huac{M}\iff \sigma(m_1)=\sigma(m_2)$;
		\item if $\{x_1,x_2\}$ is $\R$-linearly independent and $\sigma(m_1)\neq \sigma(m_2)$, then $\generat{m_1+m_2}_\O=\O m_1\oplus\O m_2.$
		
	\end{enumerate}	
%	An analogous statement holds for $\hua{C}{-}{M}$.
\end{lemma}
\begin{proof}
	Suppose 
	%$r_i\in \sigma(m_i)$ such
	  that $m_i=r_ix_i$ for $i=1,2$. 
	
	We prove a\asertion{1}. We can assume $x_1=rx_2$ for some $r\in \R$ by hypothsis, then $m_1+m_2=(r_1r+r_2)x_2\in \bigcup_{p\in \spo}p\cdot \huaa{M}$. Then the conclusion follows by Theorem \ref{lem:huac+(M)}.
	
	We prove a\asertion{2}. Clearly both  $m_1,m_2$ are nonzereo element by hypothesis.  If  $\sigma(m_1)=\sigma(m_2)$, we can assume $r_1=rr_2$ for some $r\in \R$. Hence $$m_1+m_2=rr_2x_1+r_2x_2=r_2(rx_1+x_2)\in \bigcup_{p\in \spo}p\cdot \huaa{M}.$$
	This shows that $m_1+m_2\in \huac{M}$. Now suppose $m_1+m_2\in \huac{M}$. Then we can choose $0\neq p\in\O$ such that $m_1+m_2=px$ for some $x\in\huaa{M}$. It follows that 
	$$x=(p^{-1}r_1)x_1+(p^{-1}r_2)x_2,$$
	in view of Lemma \ref{lem: xishu asselm}, we conclude $p^{-1}r_1\in \R,\; p^{-1}r_2\in\R$, which implies $\sigma(m_1)=\sigma(m_2)$.
	
	We prove a\asertion{3}. Thanks to  a\asertion{2} that we have just proved, we deduce $$\dim_\R(\generat{m_1+m_2}_\O)>8.$$
	However $\generat{m_1+m_2}_\O$ is a submodule included by $\O m_1\oplus\O m_2$, which yields $$\dim_\R(\generat{m_1+m_2}_\O)\leqslant 16.$$This forces $\generat{m_1+m_2}_\O=\O m_1\oplus\O m_2.$		
\end{proof}

%
%\begin{rem}
%	Recall in  previous work \cite{liyong2019octonionmodule},  we have put forward a conjecture about the structure of submodules.
%	For any element $m$ in a  left $\O$-module $M$,  there exist $m^\pm\in \text{Span}_\R\hua{C}{\pm}{M}$ such that $m=m^++m^-$.  We can decompose $m^{\pm}$ into  a combination of real linearly independent cyclic elements. Denote by $l_m^{\pm}$ the minimal length of the decompositions of $m^{\pm}$. {The conjecture is that,} the submodule generated by $m$ is: $$\generat{m}_\O\cong \O^{l_m^+}\oplus\O^{l_m^-}.$$
%	In the situation of $\O$-bimodules case, we can give a positive answer  as the following theorem asserts. 
%\end{rem}

Now we can describe the structure of these submodules generated by one element.
\begin{thm}\label{thm:<m>=oplus Omi}
			Let $m$ be an arbitrary element of an $\O$-bimodule $M$. Then  $$\generat{m}_\O=\bigoplus_{i=1}^n\O m_i,$$
			where  $\{m_i\}_{i=1}^n\subseteq \huac{M}$  is an arbitrary \decomp\ of $m$.

%		Let $m$ be an arbitrary element of an $\O$-bimodule $M$. there exits a \decomp\ $\{m_i\}_{i=1}^n\subseteq \huac{M}$  of $m$,
%	such that $$\generat{m}_\O=\bigoplus_{i=1}^n\O m_i.$$
	%Moreover, $\{m_i\}_{i=1}^n\subseteq \huac{M}$ is a \decomp\ of $rm$ for some $r\in \O$.

%		For any $m\in {M}$, there exits a collection of cyclic elements $\{m_i\}_{i=1}^n\subseteq \huac{M}$,
%	such that $$\generat{m}_\O=\bigoplus_{i=1}^n\O m_i.$$
%	Moreover, $\{m_i\}_{i=1}^n\subseteq \huac{M}$ is a \decomp\ of $rm$ for some $r\in \O$.
%	For any $m\in \text{Span}_\R\hua{C}{+}{M}$, there exits a \decomp\ $\{m_i\}_{i=1}^n\subseteq \hua{C}{+}{M}$ 
%	such that $$\generat{m}_\O=\bigoplus_{i=1}^n\O m_i.$$
%	\begin{enumerate}
%		\item the collection of correspoding  \tezh s $\{x_i\}_{i=1}^n$ is $\R$-linearly independent;
%		\item the collection of  $\{r_i\}_{i=1}^n$ is $\R$-linearly independent in $\O$, where $r_i\in \sigma(m_i)$;
%	\item $m=\sum _{i=1}^n m_i $ and the submodule generated by $m$ is  $$\generat{m}_\O=\bigoplus_{i=1}^n\O m_i.$$
%	\end{enumerate}
%An analogous statement holds for $\hua{C}{-}{M}$.
	
\end{thm}
\begin{rem}
	In particular, the length $n$ of a \decomp\ $\{m_i\}_{i=1}^n$ is  an invariant  of $m$, just called the  \textbf{length} of $m$, and denoted by $l_m$.
	It then follows that $\generat{m}_\O\cong\O^{l_m}$.
	Clearly the number of $\R$-linearly independent vectors in $\O$ is at most $8$, hence for any element $m$ of an $\O$-bimodule   $M$, we  infer that 
	$\dim_{\R}(\generat{m}_\O)\leqslant 64$. In terms of the concept of length, we have only $8$ kinds of elements in an $\O$-bimodule.
\end{rem}
\begin{proof}[Proof of Theorem \ref{thm:<m>=oplus Omi}]
%	By Corollary \ref{cor:M=Span C(M)}, we can assume $m=\sum_{i=1}^n m_i$ for some $m_i\in \huac{M}$. Note that $m\in \text{Span}_\R\hua{C}{+}{M}$, this yields $m_i\in \hua{C}{+}{M}$ obviously. Lemma \ref{lem:m1+m2} ensures us to suppose the collection $\{m_i\}_{i=1}^n$ satisfy a\asertion{1} and a\asertion{2} in Theorem. It remains to prove a\asertion{3}. 	The proof is by induction on $n$.  The case $n=1$ is trivial. Assume it holds for degree $k$, we will prove it for $k+1$. Clearly, $$x_{k+1}+\sum_{i=1}^k (r_{k+1}^{-1}r_i)x_i=r_{k+1}^{-1}m\in \generat{m}_\O,$$
%	This yields $$\sum _{i=1}^k[p,q,r_{k+1}^{-1}r_i]x_i\in \generat{m}_\O, \quad \forall p,q\in \O.$$
%	Denote $s_i=[p,q,r_{k+1}^{-1}r_i]$. We can choose $p,q\in \O$ such that   $s_i\neq 0$ for some $i$. 
%	%$\{[p,q,r_{k+1}^{-1}r_i]\}_{i=1}^k$
%	If not, then $[p,q,r_{k+1}^{-1}r_i]= 0$ for any $p,q\in \O$ and $i=1,\ldots,k.$ It follows that $r_{k+1}^{-1}r_i\in \R$ for $i=1,\ldots,k.$ 
%	This contradicts  a\asertion{2}. Hence we can obtain  a collection  of $\R$-linearly independent elements, and without loss of generality we can assume, it is  $\{s_i\}_{i=1}^{k_1}$. By induction hypothesis we conclude that, 
%		$$\bigoplus_{i=1}^{k_1}\O x_i=\generat{\sum_{i=1}^{k_1} s_ix_i}_\O$$
We prove this in six steps. Throughout the proof, let $\{m_i\}_{i=1}^n$ be an arbitrary  \decomp\ of $m$  with correspoding  \tezh s $x_i$ and \tezzh s $r_i$ for $i=1,\ldots,n$. One direction is obvious, it remains to show that $\O m_i\subset \generat{m}_\O$ for each $i$.
Note that the theorem holds for $n=1$ trivially and has been proved  for $n=2$ in Lemma \ref{lem:m1+m2}.
\begin{step }
	We can assume $r_1=1$.
\end{step }
	%\textbf{Step 1.} We can assume $r_1=1$.
	
	If not, replacing $m$ with $r_1^{-1}m$ and letting $r_i'=r_1^{-1}r_i,\; x_i'=x_i, m_i'=r_i'x_i'$, then neither the hypothesis nor the conclusion is affected since $\generat{m}_\O=\generat{r_1^{-1}m}_\O$.
	
\begin{step }
	 We can assume $r_i\in \pureim{\O}$ for $i=2,\ldots,n$.
\end{step }
%	\textbf{Step 2.} We can assume $r_i\in \pureim{\O}$ for $i=2,\ldots,n$.
	
	In fact, let $r_i=r_{i0}+\sum r_{ij}e_j$, where $r_{ij}\in \R$ for $j=0,1,\ldots,7$. 
	Set $$m_1'=\sum _{i=1}^nr_{i0}x_i;\qquad m_i'=(r_i-r_{i0})x_i,\;r_i'= r_i-r_{i0},\ \; i=2,\ldots,n.$$
	Clearly $r_i'\in \pureim{\O}$ for $i=2,\ldots,n$. We next show that $\{m_i'\}_{i=1}^n$ be another \decomp\ of $m$.
	Note that we have assumed that  $r_1=1$ from Step 1, we hence conclude that $$\sum _{i=1}^nm_i'=\sum _{i=1}^nr_{i0}x_i+\sum _{i=2}^n(r_i-r_{i0})x_i=m.$$ 
	Let $\sum _{i=1}^n t_ir_i'=0$ for some $t_i\in \R$. That is, $$0=t_1+\sum_{i=2}^nt_i(r_i-r_{i0})=(t_1-\sum_{i=2}^nt_ir_{i0})r_1+\sum_{i=2}^nt_ir_i.$$
	It follows from the linear independence of $\{r_i\}_{i=1}^n$ that $t_i=0$ for $i=1,\ldots,n$. Hence  $\{r_i'\}_{i=1}^n$	 is $\R$-linearly independent. So $ \{m_i'\}_{i=1}^n$ is another \decomp\ of $m$ with $r_i\in \pureim{\O}$ for $i=2,\ldots,n$. Moreover, it is easy to verify that $$\bigoplus_{i=1}^n \O m_i=\bigoplus \O m_i'.$$
	This means it does not affect the conclusion as well.

\begin{step }
We prove the case $n=3$.
\end{step }
%\textbf{Step 3.} We prove for $n=3$.

Let $\alpha$ be any element orthogonal to the associative subspace $$\Lambda(r_2,r_3)=\{x\in \pureim{\O}\mid [r_2,r_3,x]=0  \}.$$ Then $$[r_2,\alpha,m]=[r_2,\alpha,r_3]x_3\in \generat{m}_\O.$$ Since $[r_2,\alpha,r_3]\neq 0$,  we conclude that $\O x_3\subseteq\generat{m}_\O$ and hence  $r_3x_3\in \generat{m}_\O $,  this yields  $r_1x_1+r_2x_2\in \generat{m}_\O$ and it then follows from  Lemma \ref{lem:m1+m2}, i.e., the case $n=2$.

\begin{step }
We prove the case $n=4$.	
\end{step }
%\textbf{Step 4.} We prove for $n=4$.
\begin{itemize}
	\item If $r_4\in \Lambda(r_2,r_3)$.
	
	According to $r_2,r_3,r_4\in \Lambda(r_2,r_3)$ and noting the dimension of associative subspaces are all $3$, we conclude $\Lambda(r_2,r_4)=\Lambda(r_2,r_3)=\Lambda(r_3,r_4)$.  
	Choose $\alpha\in \Lambda(r_2,r_3)^{\perp}$ again, then  $$[r_2,\alpha,m]=[r_2,\alpha,r_3]x_3+[r_2,\alpha,r_4]x_4\in \generat{m}_\O.$$
	Let $t_1[r_2,\alpha,r_3]+t_2[r_2,\alpha,r_4]=[r_2,\alpha,t_1r_3+t_2r_4]=0$. If $t_1r_3+t_2r_4\neq0 $, it follows from the linear independence of $\{r_i\}_{i=2}^4$, we know $\{r_2,t_1r_3+t_2r_4\}$ is $\R$-linearly independent and hence we conclude 
	$$\alpha\in \Lambda(r_2,t_1r_3+t_2r_4)=\Lambda(r_2,r_3).$$
	This is a contradiction and hence $t_1r_3+t_2r_4=0 $. This immediately implies $t_1=t_2=0$, which means $\{[r_2,\alpha,r_3],[r_2,\alpha,r_4]\}$  is $\R$-linearly independent. It then turns to the case $n=2$.
	\item  If $r_4\notin \Lambda(r_2,r_3)$.
	
	It follows that $[r_2,r_3,m]=[r_2,r_3,r_4]x_4\in \generat{m}_\O$ and then turns to the case $n=3$.
\end{itemize}

\begin{step }
	We prove the case $n=5$.
\end{step }
%\textbf{Step 5.} We prove for $n=5$.

Obviously, it is impossible that both $r_4,r_5$ are in $\Lambda(r_2,r_3)$ since orthewise the dimension of $\Lambda(r_2,r_3)$ will exceed $3$.
\begin{itemize}
	\item If $r_4\in \Lambda(r_2,r_3)$. 
	
	We must have $r_5\notin  \Lambda(r_2,r_3)$, since $[r_2,r_3,m]=[r_2,r_3,r_5]x_5\in \generat{m}_\O,$ we deduce $r_5x_5\in \generat{m}_\O$. It then follows from the case $n=4$.
	\item  The case $r_5\in \Lambda(r_2,r_3)$ is similar. 
	\item If $r_4,r_5\notin \Lambda(r_2,r_3)$.
	
	We have $ [r_2,r_3,m]=[r_2,r_3,r_4]x_4+[r_2,r_3,r_5]x_5\in \generat{m}_\O$. Let $r_i'=[r_2,r_3,r_i]$ for $i=4,5$. If $\{r_4',r_5'\}$ is $\R$-linearly independent, then by the case of $n=2$, we conclude $\O x_4\oplus\O x_5\subseteq \generat{m}_\O$ and then by the case $n=3$, we deduce $\bigoplus_{i=1}^5\O x_i\subseteq \generat{m}_\O$. Suppose $\{r_4',r_5'\}$ is $\R$-linearly dependent. \bfs\ $r_5'=tr_4'$ for some $0\neq t\in \R$. Thus $r_4'(x_4+tx_5)\in \generat{m}_\O$, this implies $r_4(x_4+tx_5)\in \generat{m}_\O$. It follows that $$\sum _{i=1}^3r_ix_i+(r_5-tr_4)x_5\in \generat{m}_\O.$$
	Clearly $\{r_1,r_2,r_3,r_5-tr_4\}$ is also $\R$-linearly independent, hence this turns to the case $n=4$.

\end{itemize}
\begin{step }
	We prove for $n>5$.
\end{step }
%\textbf{Step 6.}  We prove for $n>5$.

%Note that $\{r_i\}_{i=2}^n\subset\pureim{\O}$ is a linearly independent set of $n-1>4$ vectors,   thanks to Lemma \ref{lem:ass sp exist lem},

Suppose there is an associative subspace spaned by $\{r_{i_k}\}_{k=1}^3$ for some $i_k\in \{2,\ldots,n\}$,
we simply assume $\Lambda=\text{Span}_\R(r_2,r_3,r_4)$ is an associative space. We claim that $\{[r_2,r_3,r_i]\}_{i=5}^n$ is $\R$-linearly independent. Indeed, since $\pureim{\O}=\Lambda\oplus\Lambda^\perp,$ we have , $$r_i=\alpha_i+\beta_i, \quad \alpha_i\in \Lambda,\;\beta_i\in \Lambda^\perp,$$ for each $i\in \{5,\ldots,n\}$. Suppose $\sum_{i=5}^nt_i[r_2,r_3,r_i]=0$ for some $t_i\in \R$, $i=5,\ldots,n$. Hence we obtain  $[r_2,r_3,\sum_{i=5}^nt_i\beta_i]=0$, which implies $\sum_{i=5}^nt_i\beta_i\in \Lambda\cap \Lambda^\perp$, and thus $\sum_{i=5}^nt_i\beta_i=\sum_{i=5}^nt_i(r_i-\alpha_i)=0$. Let $\alpha_i=t_{i2}r_2+t_{i3}r_3+t_{i4}r_4$, we conclude 
$$\sum_{i=5}^nt_ir_i-\sum_{i=5}^nt_i(t_{i2}r_2+t_{i3}r_3+t_{i4}r_4)=0. $$
In view of the linear independence of $\{r_i\}_{i=2}^n$, we deduce that $t_i=0$ for $i=5,\ldots,n$ and hence $\{[r_2,r_3,r_i]\}_{i=5}^n$ is $\R$-linearly independent as desired. By above claim, we can deduce $$\bigoplus_{i=5}^n\O x_i\subseteq\generat{m}_\O.$$ Then the rest of the proof runs as before.

Now suppose every subspace spaned by $\{r_{i_k}\}_{k=1}^3$ is not associative.  Denote $r_i'=[r_2,r_3,r_i]$ and hence $r_i'\neq 0$ for $i=4,\ldots,n$.   It follows that $\sum_{i=4}^nr_i'x_i\in \generat{m}_\O.$ 
%We  now prove this by  induction on $n$.
\begin{itemize}
	\item If $\{r_i'\}_{i=4}^n$ is $\R$-linearly independent.
	
	It follows from the case $n-3$ that $\bigoplus_{i=4}^n\O x_i\subset \generat{m}_\O$, and hence $\sum_{i=4}^nr_ix_i\in \generat{m}_\O$ which implies $\sum_{i=1}^3r_ix_i\in \generat{m}_\O$. Therefore $\bigoplus_{i=1}^n\O x_i\subset \generat{m}_\O$.
	\item  If $\{r_i'\}_{i=4}^n$ is $\R$-linearly dependent.
	
	\bfs\ $$r_n'=\sum_{i=4}^{n-1}t_ir_i',\quad  t_i\in \R\text{ for each }i.$$
	Let $x_i'=x_i+t_ix_n$. Clearly $\{x_i'\}_{i=4}^{n-1}$ is $\R$-linearly independent and 
	 $$\sum_{i=4}^nr_i'x_i=\sum_{i=4}^{n-1}r_i'x_i+\sum_{i=4}^{n-1}t_ir_i'x_n=\sum_{i=4}^{n-1}r_i'x_i'\in \generat{m}_\O.$$
	 \begin{itemize}[label=$\star$]
	 	\item If $\{r_i'\}_{i=4}^{n-1}$ is $\R$-linearly independent.
	 	
	 	It follows from the case $ n-4$ that $\bigoplus_{i=4}^{n-1}\O x_i'\subset \generat{m}_\O$, and hence
	 	 $$\sum_{i=4}^{n-1}r_ix_i'=\sum_{i=4}^{n-1}r_ix_i+\left(\sum_{i=4}^{n-1}r_it_i\right)x_n\in \generat{m}_\O.$$
	 	 Then $$m-\sum_{i=4}^{n-1}r_ix_i'=\sum_{i=1}^{3}r_ix_i+\left(r_n-\sum_{i=4}^{n-1}r_it_i\right)x_n\in \generat{m}_\O.$$
	 	 It is easy to see that $\{r_1,r_2,r_3,r_n-\sum_{i=4}^{n-1}r_it_i \}$ is $\R$-linearly independent,  according to the case $n=4$, we infer	 	 that $r_n x_n\in \generat{m}_\O$ and then reduces to the case $n-1$.
	 	 \item If $\{r_i'\}_{i=4}^{n-1}$ is $\R$-linearly dependent.
	 	 
		\bfs\ $$r_{n-1}'=\sum_{i=4}^{n-2}s_ir_i',\quad  s_i\in \R\text{ for each }i.$$
	 	 Let $x_i''=x_i'+s_ix_{n-1}'$. Clearly $\{x_i''\}_{i=4}^{n-2}$ is $\R$-linearly independent and 
	 	 $$\sum_{i=4}^{n-1}r_i'x_i'=\sum_{i=4}^{n-2}r_i'x_i'+\sum_{i=4}^{n-2}s_ir_i'x_{n-1}'=\sum_{i=4}^{n-2}r_i'x_i''\in \generat{m}_\O.$$
	 	 \begin{itemize}[label=*]
	 	 	\item If $\{r_i'\}_{i=4}^{n-2}$ is $\R$-linearly independent.
	 	 	
	 	 		It follows from the case $n-5$ that $\bigoplus_{i=4}^{n-2}\O x_i''\subset \generat{m}_\O$, and hence
	 	 	$$\sum_{i=4}^{n-2}r_ix_i''
	 	 	%=\sum_{i=4}^{n-2}r_i(x_i'+s_ix_{n-1}')
	 	 	=\sum_{i=4}^{n-2}r_ix_i+\left(\sum_{i=4}^{n-2}r_is_i\right)x_{n-1}+\left(\sum_{i=4}^{n-2}r_it_i+r_is_it_{n-1}\right)x_{n}\in \generat{m}_\O.$$
	 	 	Then $$m-\sum_{i=4}^{n-2}r_ix_i''=\sum_{i=1}^{3}r_ix_i+\left(r_{n-1}-\sum_{i=4}^{n-2}r_is_i\right)x_{n-1}+\left(r_n-\sum_{i=4}^{n-2}(r_it_i+r_is_it_{n-1})\right)x_n\in \generat{m}_\O.$$
	 	 	It is easy to see that $\{r_1,r_2,r_3,r_{n-1}-\sum_{i=4}^{n-2}r_is_i,r_n-\sum_{i=4}^{n-2}(r_it_i+r_is_it_{n-1}) \}$ is $\R$-linearly independent, it follows from the case $n=5$ that 
	 	 	$r_n x_n\in \generat{m}_\O$ and then reduces to the case $n-1$.
	 	 	Note that we have already proved the case $n=6$.
	 	 		\item If   $\{r_i'\}_{i=4}^{n-2}$ is $\R$-linearly dependent.
	 	 	
%	 	 	In this case, we must have $n=7$ and $r_4'=tr_5'$ for some $t\in \R$.
 	 	Apply the  argument similar to above twice, we can obtain an $\R$-linearly independent subset of  $\{r_i'\}_{i=4}^{n}$ since $r_i'\neq 0$ for each $i$. Thus it  can always   reduce to the preceding  case. 
	 	 \end{itemize}
	 \end{itemize}
\end{itemize}
This proves the theorem.
%In summary, we can obtain an $\R$-linearly independent subset of  $\{r_i'\}_{i=4}^{n}$ since $r_i'\neq 0$ for each $i$. Thus it  can always   reduce to the preceding  case. 
%By what we just proved, we will deduce $$\oplus_{i=5}^n\subseteq\generat{m}_\O.$$
\end{proof}

\begin{eg}
	Let $M=\O^3$. Let $m=(e_1,e_2,e_1+e_2)$, then $m=e_1(1,0,1)+e_2(0,1,1)$ and clearly this is a \decomp\ of $m$, hence $\generat{m}_\O=\O (1,0,1)\oplus \O (0,1,1)$.
	Consider the example in \cite{goldstine1964hilbert}, let $m=(e_1,e_2,e_3)$, then $m=e_1(1,0,0)+e_2(0,1,0)+e_3(0,0,1)$ and  this is a \decomp\ of $m$, hence $\generat{m}_\O=\O ^3=M$.

\end{eg}

\begin{eg}
	%	The example that will be discussed below has been also considered in \cite[Remark 4.18]{liyong2019octonionmodule}, now  it can be interpreted in terms of \decomp\ naturally.
	
	Let $M=\O^3$ and consider the element $m=(1,1+e_1,e_1)$. Then choose $m_1=(1,0,0),\;m_2=(0,1+e_1,0),\;m_3=(0,0,e_1)$ in $\huac{M}$, it clearly holds $m=\sum_{i=1}^3 m_i$ and they are real linear independent. However, this is not a  \decomp\ in the sence of definition \ref{def:decomp}.  On the other hand, we can choose 
	$m_1'=(1,1,0), \;m_2'=e_1(0,1,1)$ in $\huac{M}$, one can verify that this is indeed a \decomp\ of $m$. Thus it holds
	$$\generat{(1,1+e_1,e_1)}_\O=\O\cdot (1,1,0)\oplus \O\cdot  e_1(0,1,1)\cong \O^2.$$
	%	In fact, in the situation of $\O$-bimodules case, we can give a positive answer to the conjecture in \cite{liyong2019octonionmodule} as the following theorem asserts. 
\end{eg}

%
%An immediate consequence is about the question that how many \decomp s does an element have. 
%In order to describe it concretely, we first introduce some notations.
%Suppose  $\{m_i\}_{i=1}^n$ and  $\{m_i'\}_{i=1}^n$  are two \decomp s of $m$. Let $\xi$, $\xi'$ denote the vectors $(m_i)_{i=1}^n$ and  $(m_i')_{i=1}^n$ respectively. We call $\xi\sim\xi'$ if there exists a permutation $\sigma\in S_n$, such that $m_i=m_{\sigma(i)}'$. One can check this is indeed an equivalence. We define 
%$$D_m:=\{\xi=(m_i)_{i=1}^n\mid \{m_i\}_{i=1}^n \text{ is a \decomp\ of } m \}/\sim.$$
%%Let $\xi$, $\xi'$ denote the \decomp s   $\{m_i\}_{i=1}^n$ and  $\{m_i\}_{i=1}^n$ of $m$ respectively. We call $\xi\sim\xi'$ if there exists a permutation $\sigma\in S_n$, such that $m_i=m_{\sigma(i)}'$. 
%\begin{cor}
%	Let $m\in M$ with $l_m=n$,  then  there is a natural group structure on $D_m$, moreover $$D_m\cong GL(n,\R)/D,$$
%	where $D:=\{S=(s_{ij})\in GL(n,\R)\mid s_{ij}\neq0\iff j=\sigma(i) \text{ for some permutation }\sigma \in S_n \}$.
%\end{cor}
%\begin{proof}
%	We first give the group structure on $D_m$. Fix $\xi_0\in D_m$ arbitrarily.  
%\end{proof}
Using Theorem \ref{thm:<m>=oplus Omi}, we  point out a mistake in \cite[Lemma 2.4.2]{ludkovsky2007algebras}, which claims each element in an $\O$-bimodule will satisfy $\O x=x\O$. We shall show that only cyclic elements posses such property. 
\begin{cor}
Let $m$ be an arbitrary element in an $\O$-bimodule $M$. Then we have  $$\O m=m\O\iff m\in \huac{M}.$$
\end{cor}
\begin{proof}
	Suppose $m\in \huac{M}$, then we have $m=px$ for some $p\in \O$ and some $ x\in \huaa{M}$. It is easy to check that $$\O m=\O x=x\O=m\O.$$
	
	 On the other hand, if it holds $\O m=m\O$,  suppose on  the contrary  that $m\notin \huac{M} $. This means the length $l_m>1$, for brief, write $l_m=n$. 	
	 %We claim: 	for every \decomp\ $\{m_i\}_{i=1}^n$ of $m$, we have $m_i\notin \huaa{M}$ for each $i=1,\ldots,n$. 

	%Indeed, suppose
	Let $\{m_i\}_{i=1}^n$ be a \decomp\ of $m$,
	the corresponding \tezh s $\{x_i\}_{i=1}^n$ and \tezzh s  $\{r_i\}_{i=1}^n$ are both $\R$-linearly independent.
	% we only need to show $r_i\notin \R$ for each $i$. 
	It follows from the hypothesis $\O m=m\O$ that,  for any $p\in \O$, there is an octonion $q\in \O$ such that $pm=mq$, note that $x_i\in \huaa{M} $, we have, 
	$$0=pm-mq=\sum_{i=1}^n p(r_ix_i)-\sum _{i=1}^n(r_ix_i)q=\sum _{i=1}^nx_i(pr_i-r_iq).$$
	Since $\{x_i\}_{i=1}^n$ is $\R$-linearly independent, it follows from  Lemma \ref{lem: xishu asselm} that 
	\begin{eqnarray}\label{proof:om=mo}
	pr_i-r_iq=0
	\end{eqnarray} 
	for each $i=1,\ldots,n$.
	Fix $p$ arbitrarily, then $r_i^{-1}pr_i=q$ is a constant for all $i\in \{1,\ldots,n\}$.
	If there exists $r_i\in \R$, it  follows that this constant is  the fixed  octonion $p$ and hence
	$$ r_j^{-1}pr_j=p, \quad\text{ for each } j\neq i.$$
	Since $p$ is arbitrarily fixed, we thus obtain $r_j\in \R$ for each $j=1,\ldots,n.$ This contradicts the fact that $\{r_i\}_{i=1}^n$ is  $\R$-linearly independent. 
	
	We  can assume $r_i\notin \R$ for each $i$. Suppose $r_i\in \spc_{J_i}\setminus \R$ for some $J_i\in \S$. Substituting $p=r_j$ in \eqref{proof:om=mo}  for $j=1,\dots,n$,   we   obtain $r_ir_j=r_jr_i$ for all $i,j\in \{1,\ldots,n\}$. 
	Hence $r_i\in \cap_j\spc_{J_j}\setminus \R$ for each $i\in \{1,\ldots,n\}$. 
%	Moreover, if set $p=r_j$,  we   obtain $r_ir_j=r_jr_i$ for all $i,j\in \{1,\ldots,n\}$. 
	Consequently, there exists an imaginary unit $J\in \S$, such that $r_i\in \spc_J$ for every $i\in \{1,\ldots,n\}$. We conclude  immediately from the $\R$-linearly independence of    $\{r_i\}_{i=1}^n$  that the length $n$ is no more that $2$, it the follows from $n>1$ that  $n=2$. Suppose $r_1=a+bJ,\;r_2=c+dJ$, where $a,b,c,d\in \R$.  Let $r_1'=1,\;r_2'=J$ and $x_1'=ax_1+cx_2,\; x_2'=bx_1+dx_2$,  then $$m=r_1x_1+r_2x_2=r_1'x_1'+r_2'x_2'.$$
	The $\R$-linearly independence of    $\{x_i\}_{i=1}^2$ and $\{r_i\}_{i=1}^2$ yields the $\R$-linearly independence of    $\{x_i'\}_{i=1}^2$. It follows that $\{r_i'x_i'\}_{i=1}^2$ is  another \decomp\ of $m$ and satisfies $r_1'=1\in \R$, this contradicts the assumption above. We thus derive that $m\in \huac{M}$.
%	This will imply that   $\{r_i\} _{i=1}^n$ is  not $\R$-linearly independent, which contrdicts our assumption.
\end{proof}

\bibliographystyle{plain}

\bibliography{bimref}

\begin{thebibliography}{10}

\bibitem{baez2002octonions}
John~C. Baez.
\newblock The octonions.
\newblock {\em Bull. Amer. Math. Soc. (N.S.)}, 39(2):145--205, 2002.

\bibitem{bryant2003some}
Robert~L. Bryant.
\newblock Some remarks on {$G_2$}-structures.
\newblock In {\em Proceedings of {G}\"{o}kova {G}eometry-{T}opology
  {C}onference 2005}, pages 75--109. G\"{o}kova Geometry/Topology Conference
  (GGT), G\"{o}kova, 2006.

\bibitem{ghiloni2013slicefct}
Riccardo Ghiloni, Valter Moretti, and Alessandro Perotti.
\newblock Continuous slice functional calculus in quaternionic {H}ilbert
  spaces.
\newblock {\em Rev. Math. Phys.}, 25(4):1350006, 83, 2013.

\bibitem{goldstine1964hilbert}
H.~H. Goldstine and L.~P. Horwitz.
\newblock Hilbert space with non-associative scalars. {I}.
\newblock {\em Math. Ann.}, 154:1--27, 1964.

\bibitem{goldstine1964hilbert2}
H.~H. Goldstine and L.~P. Horwitz.
\newblock Hilbert space with non-associative scalars. {II}.
\newblock {\em Math. Ann.}, 164:291--316, 1966.

\bibitem{harvey1990spinors}
F.~Reese Harvey.
\newblock {\em Spinors and calibrations}, volume~9 of {\em Perspectives in
  Mathematics}.
\newblock Academic Press, Inc., Boston, MA, 1990.

\bibitem{horwitz1993QHilbertmod}
L.~P. Horwitz and A.~Razon.
\newblock Tensor product of quaternion {H}ilbert modules.
\newblock In {\em Classical and quantum systems ({G}oslar, 1991)}, pages
  266--268. World Sci. Publ., River Edge, NJ, 1993.

\bibitem{liyong2019octonionmodule}
Qinghai Huo, Yong Li, and Guangbin Ren.
\newblock Classification of left octonion modules.
\newblock {\em arXiv preprint arXiv:1911.08282}, 2019.

\bibitem{jacobson1954structure}
N.~Jacobson.
\newblock Structure of alternative and {J}ordan bimodules.
\newblock {\em Osaka Math. J.}, 6:1--71, 1954.

\bibitem{ludkovsky2007algebras}
S.~V. Ludkovsky.
\newblock Algebras of operators in {B}anach spaces over the quaternion skew
  field and the octonion algebra.
\newblock {\em Sovrem. Mat. Prilozh.}, (35):98--162, 2005.

\bibitem{ludkovsky2007Spectral}
S.~V. Ludkovsky and W.~Spr\"{o}ssig.
\newblock Spectral representations of operators in {H}ilbert spaces over
  quaternions and octonions.
\newblock {\em Complex Var. Elliptic Equ.}, 57(12):1301--1324, 2012.

\bibitem{ng2007quaternionic}
Chi-Keung Ng.
\newblock On quaternionic functional analysis.
\newblock {\em Math. Proc. Cambridge Philos. Soc.}, 143(2):391--406, 2007.

\bibitem{razon1992Uniqueness}
A.~Razon and L.~P. Horwitz.
\newblock Uniqueness of the scalar product in the tensor product of quaternion
  {H}ilbert modules.
\newblock {\em J. Math. Phys.}, 33(9):3098--3104, 1992.

\bibitem{razon1991projection}
Aharon Razon and L.~P. Horwitz.
\newblock Projection operators and states in the tensor product of quaternion
  {H}ilbert modules.
\newblock {\em Acta Appl. Math.}, 24(2):179--194, 1991.

\bibitem{salamon2017notesonoctonion}
Dietmar~A. Salamon and Thomas Walpuski.
\newblock Notes on the octonions.
\newblock In {\em Proceedings of the {G}\"{o}kova {G}eometry-{T}opology
  {C}onference 2016}, pages 1--85. G\"{o}kova Geometry/Topology Conference
  (GGT), G\"{o}kova, 2017.

\bibitem{schafer2017introduction}
Richard~D. Schafer.
\newblock {\em An introduction to nonassociative algebras}.
\newblock Dover Publications, Inc., New York, 1995.
\newblock Corrected reprint of the 1966 original.

\bibitem{soffer1983quaternion}
A.~Soffer and L.~P. Horwitz.
\newblock {$B^{\ast} $}-algebra representations in a quaternionic {H}ilbert
  module.
\newblock {\em J. Math. Phys.}, 24(12):2780--2782, 1983.

\bibitem{viswanath1971normal}
K.~Viswanath.
\newblock Normal operations on quaternionic {H}ilbert spaces.
\newblock {\em Trans. Amer. Math. Soc.}, 162:337--350, 1971.

\bibitem{zhevlakov1982Rings}
K.~A. Zhevlakov, A.~M. Slinko, I.~P. Shestakov, and A.~I. Shirshov.
\newblock {\em Rings that are nearly associative}, volume 104 of {\em Pure and
  Applied Mathematics}.
\newblock Academic Press, Inc. [Harcourt Brace Jovanovich, Publishers], New
  York-London, 1982.
\newblock Translated from the Russian by Harry F. Smith.

\end{thebibliography}
\end{document}